\documentclass[10pt,a4paper,oneside]{scrartcl}
\usepackage[latin1]{inputenc}
\usepackage[english]{babel}
\usepackage{amsmath}
\usepackage{amsfonts}
\usepackage{amssymb}
\usepackage{amsthm}
\usepackage{scrpage2}
\usepackage{enumerate}
\usepackage{textcomp}

\swapnumbers 
\newtheorem{thm}{Theorem}[section]
\newtheorem{prop}[thm]{Proposition}
\newtheorem{cor}[thm]{Corollary}

\newtheorem{lem}[thm]{Lemma}

\newtheorem{cl}[thm]{Claim}

\theoremstyle{definition}
\newtheorem{de}[thm]{Definition}

\newtheorem{ex}[thm]{Example}

\newtheorem{rem}[thm]{Remark}

\renewenvironment{proof}{\par\noindent {\em Proof: }}{\hfill$\Box$\medskip}
\theoremstyle{plain}

\newcommand{\cO}{\ensuremath{\mathcal{O}} }
\newcommand{\cA}{\ensuremath{\mathcal{A}} }
\newcommand{\cB}{\ensuremath{\mathcal{B}} }
\newcommand{\cM}{\ensuremath{\mathcal{M}} }

\newcommand{\cN}{\ensuremath{\mathcal{N}} }
 
\newcommand{\cS}{\ensuremath{\mathcal{S}} }
\newcommand{\cT}{\ensuremath{\mathcal{T}} }

\newcommand{\cU}{\ensuremath{\mathcal{U}} }
\newcommand{\cL}{\ensuremath{\mathcal{L}} }

\newcommand{\fp}{\ensuremath{\mathfrak{p}} }

\newcommand{\abs}{\ensuremath{{|\,.\,|}} }
\newcommand{\ch}{\ensuremath{\textrm{char}} }

\newcommand{\ring}{\ensuremath{\textrm{\upshape{ring}}} }

\newcommand{\gcdi}{\ensuremath{\textrm{gcd}} }

\newcommand{\Th}{\ensuremath{\textrm{Th}} }

\newcommand{\C}{\ensuremath{\mathbb{C}} }
\newcommand{\N}{\ensuremath{\mathbb{N}} }
\newcommand{\Q}{\ensuremath{\mathbb{Q}} }
\newcommand{\R}{\ensuremath{\mathbb{R}} }
\newcommand{\Z}{\ensuremath{\mathbb{Z}} }

\automark[subsection]{section}

\author{Katharina Dupont \\
katharina.dupont@uni-konstanz.de
	\footnote{key words and phrases: valuations; definable valuations; $q$-henselian valued fields; t-henselian topologies\newline
		MSC classes: 03C40 03C60 12J10 12L12} 
}
\title{Definable Valuations Induced by Definable Subgroups}
\begin{document}
\maketitle
\begin{abstract}
In \cite{Ko} Koenigsmann shows that every field that admits a t-hen\-selian topology is either real closed or separably closed or admits a definable valuation inducing the t-henselian topology. To show this Koenigsmann investigates valuation rings induced by certain (definable) subgroups of the field. The aim of this paper, based on the authors PhD thesis \cite{Du2015},  is to look at the methods used in \cite{Ko} in greater detail and correct a mistake in the original paper based on \cite{JaKo2012}.
\end{abstract}

\addsec{Introduction}
In this paper we will show that any non-real closed, non-separably closed field $K$, which admits a t-henselian topology, admits a non-trivial definable valuation (see Theorem~\ref{thmDefValtHenselian}). Our main tool will be to construct valuation rings using  subgroups of $K$. More precisely we will treat simultaneously additive subgroups of $K$ and multiplicative subgroups of $K^\times$.

This paper arose as follows.
Motivated by recent considerations on definable valuations under model theoretic assumptions the author reconsidered in her PhD thesis, \cite{Du2015}, an unpublished preprint of Koenigsmann, see \cite{Ko}.
This paper is mainly a revised version of the preprint. In Proposition~\ref{propDefinableWeakCase}, using \cite{JaKo2012}, we will give an alternative proof for one case in \cite[Theorem~3.1]{Ko} for which the original proof was incorrect. 
Corollary~\ref{corVAxiomsnontrivialdefinable} provides the crucial idea for the model theoretic investigation, which will be pursued in a forthcoming paper, see \cite{DuHaKu2016}.

The research on definable valuations has been very active lately. Recent works include \cite{AnKo2013} and \cite{Fe2013} on the complexity of the formulas defining valuations. In \cite{JaKo2012} conditions are given under which a definable valuation is henselian. Further  \cite{JaKo2014}, \cite{ClDeLeMa2013} and \cite{FePr2015} deal with uniformly definable valuation rings. As well \cite{JaSiWa2015} and \cite{Jo2015} on dp-minimal fields, include sections on definable valuations.

The paper is organized as follows. 
\nopagebreak

We will start with some preliminaries on fractional ideals on valued fields, topologies induced by valuations and absolute values and discrete valuations, that we will refer to later on. 

In Section~2, for every additive or multiplicative subgroup of a field $K$ we will define the valuation ring $\cO_G$ and prove some of its basic properties. 

In Section~3 we will give criteria under which $\cO_G$ is non-trivial.

In Section~4 we will examine under which criteria $\cO_G$ is definable. 

In Section~5 we will bring together the results of the previous two sections for the group of $q$-th powers $\left(K^\times\right)^q$ for $q\neq\ch(K)$ and for the Artin-Schreier group $K^{(p)}$ for 
$p=\ch(K)$. That way in Theorem~\ref{thmphenseliandefinable} we will show  that (under additional assumptions)  if  $K$ admits a non-trivial  $q$-henselian valuation for some prime $q$, then it admits a non-trivial definable valuation. From this we will finally establish Theorem~\ref{thmDefValtHenselian} on t-henselian fields as announced at the beginning of the Introduction.

\textbf{Notation:} In this paper $K$ will always denote a field and $\cO$ a valuation ring on $K$ with $\cM$ its maximal ideal. By $\varrho: K\longrightarrow \cO/\cM=:\overline{K}$ we denote the residue homomorphism. By $v$ we will denote a valuation on $K$ and by $\cO_v:=\left\{x\in K\mid v(x)\geq 0\right\}$ the valuation ring induced by $v$ with maximal ideal $\cM_v$. A valuation will be called discrete, if its value group contains a minimal positive element. Without loss of generality, we shall assume that $\Z$ is a convex subgroup of the value group and hence $1$ is the minimal positive element.

Some of the following definitions and theorems will be slightly different for additive and multiplicative subgroups. Often we will write the differences for multiplicative subgroups in square brackets ``$\left[\ldots\right]$'' if there is no danger of misunderstanding. If we say $G$ is a subgroup of $K$, this can mean either a subgroup of the additive group $(K,+)$ or the multiplicative group $(K^\times,\cdot)$, unless explicitly otherwise noted. We will say $G$ is a proper subgroup of $K$ if $G\subsetneq K$ $\left[\textrm{resp. }G\subsetneq K^\times\right]$.

I want to thank Franziska Janke for pointing out the mistake in \cite{Ko} as well as for several helpful discussions and comments on an early version of this work. Further I want to thank Salma Kuhlmann and Assaf Hasson for great support and helpful advice while I was conducting the research as well as while I was writing the paper.

\section{Preliminaries}

The following can be shown by simple calculation. 
\begin{rem}\label{remfracIdiff}
	Let $v:K\twoheadrightarrow \Gamma\cup \left\{\infty\right\}$ be a valuation. Let $\{0\}\subsetneq \cA\subsetneq K$.
	\begin{enumerate}[(a)]
		\item $\cA$ is a fractional ideal of $\cO_v$ if and only if for every  $x\in K$, if there exists $a\in \cA$ such that $v\left(x\right)\geq v\left(a\right)$, then $x\in \cA$. 
		\item The fractional ideals of $\cO_v$ are linearly ordered, i.e. if $\cA_1$ and $\cA_2$ are fractional ideals of $\cO_v$ then $\cA_1\subseteq \cA_2$ or $\cA_2\subseteq \cA_1$.  
		\item 	Let $\cA\subsetneq\cO_v$.
		$\cA$ is a prime ideal of $\cO_v$ if and only if for every  $x\in \cO_v$, if there exists $a\in \cA$ and an $n\in \N$ with $n\cdot v\left(x\right)\geq v\left(a\right)$, we have $x\in \cA$.
	\end{enumerate}
\end{rem}

\begin{lem}\label{lemMsubsetneqintersection}
	Let $\cO_2\subsetneq \cO_1$ be two valuation rings on $K$ with maximal ideals $\cM_1$ and $\cM_2$. Let $\cA$ be an $\cO_2$-ideal with $\sqrt{\cA}=\cM_2$. 
	Then $\cM_1\subsetneq \cA$.
\end{lem}
\begin{proof}
	Suppose $\cA\subseteq \cM_1$. Then $\cM_2=\sqrt{\cA}\subseteq \cM_1$. But this contradicts $\cO_2\subsetneq\cO_1$. 
	Hence $\cM_1\subsetneq \cA$ by Remark~\ref{remfracIdiff}~(b).
\end{proof}

\begin{lem}\label{lemInverseOf1plusIdeal}
	Let $\cO$ be a valuation ring and $\cA$ an $\cO$-ideal.
	Then $\left(1+\cA\right)$ is a multiplicative subgroup of $\cO^\times$. 
\end{lem}

\begin{proof}
	It is clear that $1+\cA\subseteq \cO^\times$.
	Let $a\in \cA$. 
	Then $1+a\in \cO^\times$. Hence $\left(1+a\right){-1}\in \cO^\times$. Therefore $a\cdot\left(1+a\right){-1}\in \cA$ and hence  $\left(1+a\right){-1}=1+a\cdot\left(1+a\right){-1}\in 1+\cA$.
	Further for $a,b \in \cA$ we have  $\left(1+a\right)\cdot\left(1+b\right)= 1+a+b+ a\cdot b\in 1+\cA$. 
\end{proof}

\begin{lem}\label{lemFSumProdIdeals}
	Let $v_1,\,v_2$ be independent valuations on  $K$. 
	Let $\cA_1$ be a non-trivial $\cO_{v_1}$-ideal and $\cA_2$ a non-trivial
	$\cO_{v_2}$-ideal.
	Then $K=\cA_1+\cA_2$ 
	and $K^{\times}=\left(1+\cA_1 \right)\cdot\left(1+\cA_2\right)$.
\end{lem}

\begin{proof}
	Let $b_1, b_2\in K$ and $c_1, c_2\in K^{\times}$.	
	From the Approximation Theorem (see \cite[Theorem~2.4.1]{EnPr2005}) follows with Remark~\ref{remfracIdiff} 
	$\left(b_1+c_1\cdot\cA_1\right)\cap\left(b_2+c_2\cdot \cA_2\right)\neq \emptyset$.

	Let $x\in K$. With $b_1=x,\, c_1=-1,\, b_2=0$ and $c_2=1$ follows that there exist $a_1\in \cA_1$ and $a_2\in \cA_2$ such that $x-a_1= a_2$. Thus $x= a_1+a_2\in \cA_1+\cA_2$.
	Therefore $K=\cA_1+\cA_2$.
	
	Now  let $x\in K^{\times}$. Then  with $b_1= c_1=x$ and $b_2=c_2=1$ follows that  there exist $a_1\in \cA_1$ and $a_2\in \cA_2$ such that $\displaystyle x+x\cdot a_1=1+a_2$.
	We have
	$x=\left(1+a_1\right)^{-1}\cdot\left(1+a_2\right) \in \left(1+\cA_1\right)^{-1}\cdot\left(1+\cA_2\right)=\left(1+\cA_1\right)\cdot\left(1+\cA_2\right)$ by Lemma~\ref{lemInverseOf1plusIdeal}.
	Hence $K^\times=\left(1+\cA_1 \right)\cdot\left(1+\cA_2\right)$.
\end{proof}

\begin{lem}\label{lemNonComparableRings}
	Let $\cO_1$ and $\cO_2$ be two non-comparable valuation rings on a field $K$. 
	Let $\cO$ be
	the finest common coarsening of $\cO_1$ and $\cO_2$ and $\cM$ the maximal ideal of $\cO$.
	Let $\cA_1$ be an  $\cO_1$-ideal with $\cM\subsetneq \cA_1$ and $\cA_2$ an $\cO_2$-ideal with $\cM\subsetneq \cA_2$.
	Then $\cO=\cA_1+\cA_2$ and $\cO^{\times}=\left(1+\cA_1 \right)\cdot\left(1+\cA_2\right)$.
\end{lem}

\begin{proof}
	Apply Lemma~\ref{lemFSumProdIdeals} to the valuation rings $\overline{\cO_1}$ and $\overline{\cO_2}$ induced by $\cO_1$ and $\cO_2$ on $\overline{K}=\cO/\cM$. 
\end{proof}

\begin{lem}\label{lemSubgroupsGeneratedByIdeals}
	Let $\cA$ be an $\cO$-ideal.
	
	\begin{enumerate}[(a)]
		\item Let $x\in K^{\times}$ such that $x^{-1}\notin \cA$. 
		Then for every $0\neq a\in \cA$ we have
		$\left(x-a^{-1}\right){-1}\in \cA$. 
		\item The multiplicative group generated by the non-zero elements of $\cA$ 
		is $K^{\times}$.
	\end{enumerate}
\end{lem}

\begin{proof}
	\begin{enumerate}[(a)]
		\item Let $0\neq a\in \cA$.  Let $x\in K^{\times}$ with $x^{-1}\notin \cA$.  Let $v$ be a valuation with $\cO=\cO_v$. By Remark~\ref{remfracIdiff} follows $v\left(x^{-1}\right)<v\left(a\right)$ and therefore $v\left(x\right)>v\left(a^{-1}\right)$. Hence $v\left(x-a^{-1}\right){=}v\left(a^{-1}\right)$ and therefore $v\left(\left(x-a^{-1}\right){-1}\right)=v\left(a\right)$. Again by Remark~\ref{remfracIdiff} follows $\left(x-a^{-1}\right){-1}\in \cA$.
		\item Let $0\neq x\in \cO$. Let $0\neq a\in \cA$. Then $a\cdot x\in \cA$. Therefore 
		$
		x=a^{-1}\cdot a\cdot x
		$ is contained in the multiplicative group generated by the non-zero elements of $\cA$.
		
		For $x\notin \cO$ we have $x^{-1}\in \cO$. Therefore as shown above $x^{-1}$ and hence as well $x$ is contained in the multiplicative group generated by the non-zero elements of $\cA$.
	\end{enumerate}
\end{proof}

\begin{lem} \label{lemVTopoplogy}
	Let $K$ be a field and $\mathcal{N}\subseteq \mathcal{P}\left(K\right)$ such that
	\begin{enumerate}[(V\,1)]
		\item  $\bigcap \mathcal{N}:=\bigcap_{U\in\mathcal{N}}U=\left\{0\right\}$ and $\left\{0\right\}\notin\mathcal{N}$;
		\item $\forall\, U,\,V\in\mathcal{N}\ \exists\, W\in\mathcal{N}\   W\subseteq U\cap V$;
		\item $\forall\, U\in\mathcal{N}\   \exists\,  V\in\mathcal{N}\   V-V\subseteq U$;
		\item$\forall\, U\in\mathcal{N}\   \forall\, x,\,y\in K\   \exists\, V\in\mathcal{N}\   \left(x+V\right)\cdot\left(y+V\right)\subseteq x\cdot y+U$;
		\item $\forall\, U\in \mathcal{N}\   \forall\, x\in K^{\times}\  \exists\, V\in \mathcal{N}\   \left(x+V\right)^{-1}\subseteq x^{-1}+U$;
		\item $\forall\, U\in\mathcal{N}\   \exists\, V\in \mathcal{N}\   \forall\, x,\,y\in K\   x\cdot y\in V\longrightarrow x\in U\, \vee\,  y\in U$.
	\end{enumerate}
	Then
	$\displaystyle
	\cT_{\cN}:=\left\{U\subseteq K\mid \forall\, x\in U\: \exists\, V\in\cN\:  x+V\subseteq U\right\}$
	is a Topology on $K$.  
	
	$\cN$ is a basis of zero neighbourhoods of $\cT_\cN$. 
\end{lem}

\begin{de}
	A topology such that (V\,1) to (V\,6) hold for the set of neighbourhoods of zero, is called V-topology.
\end{de}

\begin{rem}\label{remVtopifBasisV1V6}
	By \cite[Theorem~1.1]{PrZi1978} (V\,1) to (V\,6) hold for the set of neighbourhoods of zero if and only if they hold for any basis of the neighbourhoods of zero.
\end{rem}

The following was first shown in \cite{DuKo1953}. A proof can be found in \cite[Appendix~B]{EnPr2005}.

\begin{thm}\label{thmVtopiff}
	A topology is a V-topology if and only if it is induced by a non-trivial valuation or by a non-trivial absolute value.
\end{thm}

A detailed proof of the following claim can be found in \cite[Claim~3.8]{Du2015}. As it is very technical and of not much interest for the rest of the paper, we will only give a brief idea of the proof here.

\begin{prop}\label{propgroupOpenforarchAbs}
	Let $K$ be a field and $\abs$ an archimedean absolute value on $K$. 
	\begin{enumerate}[(a)]
		\item Let $G$ be an additive subgroup of $K$.
		If $G$ is open with respect to $\abs$, then $G=K$.
		\item Let $G$ be a multiplicative subgroup of $K^\times$.
		If $G$ is open with respect to $\abs$, then either $G=K^\times$ or $G\cup\{0\}$ is an ordering 
		on $K$. 
	\end{enumerate}
\end{prop}

\begin{proof}[idea]
	As any field which admits an archimedean absolute value embeds into $\R$ or $\C$, we can assume without loss of generality $K\subseteq \R$ or $K\subseteq \C$. 
	
	If $G$ is open, it contains an open neighbourhood $U$ of $0$ $\left[\textrm{resp. }1\right]$. As $G$ is closed under addition [resp. multiplication] for any $g\in G$ $g+U$ $\left[\textrm{resp. }g\cdot U\right]$ is still contained in $G$. By recursively approximating all the elements of $K$ [resp. $K^\times$ or $K^{>0}$ if $K$ is an ordered field with $g>0$ for all $g\in G$], we show the claim.
\end{proof}

The following lemma is well known. A proof can be found for example in \cite[Claim~A.43]{Du2015}.

\begin{lem}\label{lemDiscreteValuations} 
	Let $v:K\twoheadrightarrow \Gamma\cup\left\{\infty\right\}$ be a discrete valuation on $K$.
	\begin{enumerate}[(a)]
		\item Let $x\in K$. Then $x\cdot\cO_v=\cM_v$ if and only if $v\left(x\right)=1$.
		In particular there exists $x\in K$ with $x\cdot\cO_v=\cM_v$.
		\item Let $x\in K^\times$ such that $v(x)=1$. Then  for every $y\in K^\times$ with $v\left(y\right)\in \mathbb{Z}$ there exists $z\in \cO_v^{\times}$ such that $y=x^{v\left(y\right)}\cdot z$.  
	\end{enumerate}
\end{lem}

\begin{prop}\label{propMaximalValuationorBasis}
	Let $\cO$ be a non-trivial valuation ring on a field $K$. 
	\begin{enumerate}[(a)]
		\item If $\widetilde{\cO}$ is a maximal non-trivial coarsening of $\cO$, then $\widetilde{\cO}$ has rank-1.
		\item If there exists no maximal non-trivial coarsening of $\cO$, then the non-zero prime ideals of $\cO$ form a basis of the neighbourhoods of zero of the topology $\cT_\cO$. 
	\end{enumerate} 
\end{prop}

Proposition~\ref{propMaximalValuationorBasis} is a shortened version of \cite[Proposition~2.3.5]{EnPr2005}.

\section{The Valuation Ring $\cO_G$ Induced by a Subgroup $G$}\label{chDefOG}

In this section for every (additive or multiplicative) subgroup $G$ of a field, we want to define a valuation ring $\cO_G$. For this valuation we will first define when a valuation is coarsely compatible with a subgroup. We will define $\cO_G$ as the intersection over all valuation rings that are coarsely compatible with $G$. Before we will come to the definition we will prove some lemmas that we will need to show that with this definition $\cO_G$ is a valuation ring. We will conclude the section with defining three cases that will reappear in the subsequent sections.

\begin{de}\label{deCompatible}
	Let $G$ be a subgroup of $K$.  
	\begin{enumerate}[(a)]
		\item $\cO$ is \emph{compatible} with $G$ if and only if $\cM\subseteq G$ $\left[\textrm{resp. }1+\cM\subseteq G\right]$.
		\item $\cO$ is \emph{weakly compatible} with $G$ if and only if there exists an  $\cO$-ideal $\cA$ with
		$\sqrt{\cA}=\cM$ such that $\cA\subseteq G$ $\left[\textrm{resp. }1+\cA\subseteq G\right]$.
		\item $\cO$ is \emph{coarsely compatible} with $G$ if and only if $v$ is weakly compatible with $G$ and there is no proper coarsening $\widetilde{\cO}$ of $\cO$ such that $\widetilde{\cO}^{\times}\subseteq G$.
	\end{enumerate}
	Let $v$ be a valuation on $K$. 
	We call $v$ \emph{compatible} (respectively  \emph{weakly compatible}, \emph{coarsely compatible}) with $G$ if and only if $\cO_v$ is compatible  (respectively weakly compatible, coarsely compatible) with $G$.
	
	We omit "with $G$" whenever the context is clear.
\end{de}

\begin{rem}
	If $\cO^\times\subseteq G$, then $\cO$ is compatible. 
	Further if $G$ is an additive group, then $\cO\subseteq G$.
\end{rem}

\begin{proof}
	If $G$ is an additive group, $-1\in \cO^\times\subseteq G$ and hence $\cM=1+\cM-1\subseteq \cO^\times-1\subseteq G$. Hence $\cO$ is compatible and $\cO\subseteq G$.
	
	If $G$ is a multiplicative subgroup, then $1+\cM\subseteq \cO^\times\subseteq G$ and hence $\cO$ is compatible.
\end{proof}

\begin{lem}\label{lemWeaklyCompthen1+pnMvsubsetT}
	Let $\ch\left(\cO_v/\cM_v\right)=q$. Let $G$ be a subgroup of $K$. 
	Let $v$ be weakly compatible.
	Then there exists $n\in \mathbb{N}$ such that $q^n\cdot \cM_v\subseteq G$ $\left[\textrm{resp. }1+q^n\cdot \cM_v\subseteq G\right]$. 
\end{lem}
\begin{proof}
	Let $\cA$ be an $\cO_v$-ideal with $\cA\subseteq G$ $\left[\textrm{resp. }1+\cA\subseteq G\right]$ and $\sqrt{\cA}=\cM_v$.  As $q\in \cM_v$ there exists $n\in \mathbb{N}$ such that $q^n\in \cA$. Let $x\in q^n\cdot \cM_v$. Then $v(x)>v(q^n)$ and therefore by Remark~\ref{remfracIdiff}~(a) $x\in \cA$. Hence $q^n\cdot \cM_v\subseteq G$ $\left[\textrm{resp. }1+q^n\cdot \cM_v\subseteq 1+\cA\subseteq G\right]$. 
\end{proof}

\begin{lem}\label{lemCoarselyCompatibleComparable}
	Let $G$ be a subgroup  of a field $K$.
	Then any two coarsely compatible  valuation rings are comparable.
\end{lem}

\begin{proof}
	Let $\cO_1$ and $\cO_2$ be two weakly compatible  valuation rings on $K$. For $i=1,2$ let $\cM_i$ be the maximal ideal of $\cO_i$ and  $\cA_i$  $\cO_{i}$-ideals with $\cA_i\subseteq G$ $\left[\textrm{resp. }1+\cA_i\subseteq G\right]$ and $\sqrt{\cA_i}=\cM_{i}$.  Suppose $\cO_1$ and $\cO_2$ are not comparable. Let $\cO$ be the finest common coarsening of $\cO_1$ and $\cO_2$. Let $\cM$ be the maximal ideal of $\cO$. 
	From Lemma~\ref{lemMsubsetneqintersection} follows that $\cM\subsetneq \cA_1$ and $\cM\subsetneq \cA_2$. By Lemma~\ref{lemNonComparableRings} we have 
	$\cO^{\times}\subseteq \cA_1+\cA_2\subseteq G$ $\left[\textrm{resp. }\cO^{\times}=\left(1+\cA_1\right)\cdot \left(1+\cA_2\right)\subseteq G\right]$. 
	Hence 
	by definition $\cO_1$ and $\cO_2$ are not coarsely compatible.
\end{proof}

Set $\cO_G:=\bigcap\left\{\cO\mid \cO\textrm{ coarsely compatible with }G\right\}$.
\begin{thm}\label{thmOG}
	\begin{enumerate}[(a)]
		\item $\cO_G$ is a valuation ring on $K$.
		\item $\cO_G$ is coarsely compatible.
	\end{enumerate}
\end{thm}

\begin{proof}
	\begin{enumerate}[(a)]
		\item This follows from Lemma~\ref{lemCoarselyCompatibleComparable}.
		\item Let  $\mathcal{C}:=\left\{\cO\mid \cO\textrm{ coarsely compatible with }G\right\}$. For every $\cO \in \mathcal{C}$ let $\cM_\cO$ be the maximal ideal of $\cO$ and let $\cA_\cO$ be an $\cO$-ideal with $\sqrt{\cA_\cO}=\cM_\cO$ and $\cA_\cO\subseteq G$ $\left[\textrm{resp. }1+\cA_\cO\subseteq G\right]$. Define 
		$\cA_G:=\bigcup\left\{\cA_\cO\mid \cO\in \mathcal{C}\right\}.$ Let $\cM_G$ be the maximal ideal of $\cO_G$. 
		
		Let $a,b\in \cA_G$ and $x\in \cO_G$. 
		There exist $\cO_1, \cO_2\in \mathcal{C}$ such that $a\in\cA_{\cO_1}=:\cA_1$ and $b\in\cA_{\cO_2} =:\cA_2$. 
		By Lemma~\ref{lemCoarselyCompatibleComparable}  let without loss of generality $\cO_1\subseteq \cO_2$. 
		Then $\cA_2\subseteq \cA_1$ and therefore $a,b\in \cA_1$. As $\cA_1$ is an ideal $a+b\in \cA_1\subseteq \cA_G$. 
		Further $x\in \cO_G$ and therefore  $x\in \cO_1$. Therefore $x\cdot a\in \cA_1\subseteq \cA_G$.
		For every valuation $ \cO\in \mathcal{C}$   $\cA_\cO\subseteq \cM_{\cO}\subseteq \cM_G$. Hence $\cA_G\subseteq \cM_G$ and thus $\sqrt{\cA_G}\subseteq \cM_G$. 
		
		On the other hand let $x\in \cM_G$. It is easy to see that there exists $\cO\in \mathcal{C}$  such that $x\in \cM_\cO=\sqrt{\cA_\cO}$. Therefore there exists an $n\in \mathbb{N}$ such that $x^n\in \cA_\cO\subseteq \cA_G$ and hence $x\in \sqrt{\cA_G}$. Therefore $\cM_G\subseteq \sqrt{\cA_G}$. 
		As  $\cA_{\cO}\subseteq G$ 	$\left[\textrm{resp. }1+\cA_{\cO}\subseteq G\right]$ for every $\cO\in \mathcal{C}$  we have $\cA_G\subseteq G$ $\left[\textrm{resp. }1+\cA_G\subseteq G\right]$.
		Hence $\cO_G$ is  weakly compatible. 
		
		Assume $\cO_G$ is not coarsely compatible. Let $\cO$ be a valuation ring such that $\cO_G\subsetneq \cO$ and $\cO^\times \subseteq G$. Without loss of generality let  $\cO$ be coarsely compatible. 
		Let $x\in \cO\setminus \cO_G$. Then there exists a valuation ring $\widetilde{\cO}\in \mathcal{C}$  with $x\notin \widetilde{\cO}$. By Lemma~\ref{lemCoarselyCompatibleComparable} $\widetilde{\cO}$ and $\cO$ are comparable. As $x\in \cO\setminus \widetilde{\cO}$ we have $\widetilde{\cO}\subsetneq \cO$. But this is contradicts $\widetilde{\cO}$  coarsely compatible. 
		This shows that $\cO_G$ is coarsely compatible. 
	\end{enumerate}
\end{proof}

\begin{de}
	We call $\cO_G$ the valuation ring induced by $G$.
\end{de}

In the whole paper 	let $\cM_G$  denote the maximal ideal of $\cO_G$ and let $v_G$ be a valuation with $\cO_{v_G}=\cO_G$.

\begin{thm}\label{corCaseDistinction}
	For any  subgroup $G$ of a field $K$ one of the following cases holds:
	\begin{description}
		\item [group case] \emph{There is a valuation ring $\cO$ with $\cO^\times\subseteq G$.}
		
		In this case $\cO_G$ is the only  coarsely compatible valuation ring with this property.
		All weakly compatible valuations  are compatible.
		\item [weak case] \emph{There exists a weakly compatible valuation ring which is not compatible.} 
		
		In this case $\cO_G$ is the only valuation ring with this property.
		\item [residue case] \emph{All weakly compatible valuations are compatible and there is no valuation ring $\cO$ with $\cO^{\times}\subseteq G$.}
		
		In this case $\cO_G$ is the finest compatible valuation ring. 
	\end{description}
\end{thm}

\begin{proof}
	\begin{description}
		\item [group case] Let $\cO$ be a valuation ring  with $\cO^\times\subseteq G$.
		Let\\
		$\widetilde{\cO}:=\bigcup \left\{\cO\mid \cO\textrm{ valuation ring such that }  \cO^{\times}\subseteq G\right\}$.
		Let $x, y\in \widetilde{\cO}$. Then there exist $\cO_1, \, \cO_2\in\left\{\cO\mid  \cO^{\times}\subseteq G\right\} $ such that $x\in \cO_1$ and $y\in \cO_2$. If $\cO_1$ and $\cO_2$ are comparable $x+y,\,x\cdot y\in \widetilde{\cO}$ is clear.
		Otherwise let $\cO$ be the finest common coarsening of $\cO_1$ and $\cO_2$. By Lemma~\ref{lemNonComparableRings},   $\cO=\cM_1+\cM_2\subseteq G$ $\left[\textrm{resp. }\cO^\times=\left(1+\cM_1\right)\cdot\left(1+\cM_2\right)\subseteq G\right]$.
		As $x, y\in \cO$ we have $x+y, x\cdot y\in \cO$ and therefore $x+y, x\cdot y\in \widetilde{\cO}$. 
		Further if $x\in \widetilde{\cO}$ then $x\in \cO$ for some valuation ring $\cO$ such that $\cO^\times\subseteq G$. Hence $-x\in \cO\subseteq \widetilde{\cO}$. 
		Hence $\widetilde{\cO}$ is a ring.
		By assumption it is clear that $\widetilde{\cO}$ is a valuation ring. 
		
		Now let $x\in \widetilde{\cO}^\times$.  As above we can find a valuation ring $\cO$ such that $x,\, x^{-1}\in \cO$ and $\cO^\times\subseteq G$. Hence $\widetilde{\cO}^\times \subseteq G$.
		Further by  definition, $\widetilde{\cO}$ is coarsely compatible.
		Hence $\cO_G\subseteq \widetilde{\cO}$. As $\cO_G$ is by Theorem~\ref{thmOG}~(b) coarsely compatible, it follows that $\cO_G=\widetilde{\cO}$. 
		In particular $\cO_G$ is compatible. 
		
		By Lemma~\ref{lemCoarselyCompatibleComparable} follows that there can be at most one coarsely compatible valuation ring $\cO$ with $\cO^\times\subseteq G$.  
		
		Let $\cO$ be weakly compatible.
		
		If $\cO^{\times}\subseteq G$ then $\cO$ is compatible. 
		
		If $\cO^\times \not \subseteq G$  we have $\cO_G\subsetneq \cO$. Hence $\cM\subseteq \cM_G$ and therefore  $\cO$ is compatible.
		
		\item [weak case]Let  $\cO$ be 
		weakly compatible but not compatible. 
		
		By the group case $\cO^\times\not\subseteq G$. Hence $\cO$ is coarsely compatible and therefore $\cO_G\subseteq \cO$.
		From Lemma~\ref{lemMsubsetneqintersection} follows $\cO_G=\cO$ as otherwise $\cO$ would be compatible. 
		
		\item[residue case]$\cO_G$ is the finest coarsely compatible valuation ring. By assumption in the residue case the coarsely compatible valuation rings are exactly the compatible valuation rings.  
	\end{description}	
\end{proof}

In the group case the $\cO_G^{\times}$, and in the additive case even $\cO_G$, is contained in the subgroup. In the residue case $G$ induces a proper subgroup on the residue field $\cO_G/\cM_G$. In Section~\ref{secDefinabilityASandPowers}, when proving the definability of $\cO_G$ under certain conditions, in the residue case for part of the proof we will be working in the residue field. The name weak case does not need any further motivation.


\section{Criteria for the Non-Triviality of $\cO_G$}
\label{chNonTriviality}

In the whole section let $G\subseteq K$ $\left[\textrm{resp. }G\subseteq K^\times\right]$ be a  subgroup of  $K$.

The valuation ring $\cO_G$, that we have defined in the last section, is not necessarily non-trivial. In this section we will give criteria under which $\cO_G$ is non-trivial. In particular we will show that we can express the non-triviality of $\cO_G$ in a suitable first order language.  

\begin{lem}\label{lemNonTrivialIFF}
	$\cO_G$ is non-trivial if and only if $G\neq K$ $\left[\textrm{resp. }G\neq K^{\times}\right]$ and there exists a non-trivial  weakly compatible valuation.
\end{lem}

\begin{proof}
	Assume that  $G\neq K$ $\left[\textrm{resp. }G\neq K^{\times}\right]$ and $\cO$ is a non-trivial weakly compatible valuation ring.
	
	If we are in the group case we have $\cO_G\subseteq G\subsetneq K$ $\left[\textrm{resp. }\cO_G^{\times} \subseteq G\subsetneq K^{\times}\right]$ and therefore $\cO_G$ non-trivial.
	
	If we are in the weak case  $\cM_G\not \subseteq G$ $\left[\textrm{resp. }1+\cM_G\not\subseteq G\right]$. Hence $\cM_G\neq \left\{0\right\}$ and thus $\cO_G$ is non-trivial.
	
	In the residue case  we have $\cO_G\subseteq \cO\subsetneq K$ and hence $\cO_G$ is non-trivial. 
	
	Conversely assume $\cO_G\subsetneq K$ is non-trivial. 
	Then $\cO_G$ is a non-trivial weakly compatible valuation ring.
	
	Further suppose $G= K$ $\left[\textrm{resp. }G= K^{\times}\right]$. For the trivial valuation $\cO_{tr}=K$ we have $\cO_{tr}^\times \subseteq G$. Therefore  no non-trivial valuation can be coarsely compatible.
\end{proof}

\begin{de}\label{decTG}
	We denote the coarsest topology for which $G$ is open and for which M\"obius transformations [resp. linear transformations] are continuous, by $\cT_G$. We call $\cT_G$ the topology induced by $G$.
\end{de}

\begin{thm}\label{thmcTG} 
	Let
	\[
	\cS_G:=\left.\left\{\left\{\left.\frac{a\cdot x+b}{c\cdot x+d}\ \right|\  x\in G, c\cdot x\neq -d\right\}\ \right|\  a,b,c,d \in K, a\cdot d-b\cdot c\neq 0\right\}
	\]
	\[
	\left[\textrm{resp. }\cS_G:=\left\{\left.a\cdot G+b\ \right|\  a\in K^{\times}, b\in K\right\}\right].
	\] 
	Then $\cS_G$ is a subbase of this topology.
\end{thm}

\begin{proof}
	As  $G\in \cS_G$  it is open in the topology induced by $\cS_G$.
	
	The inverse functions and compositions of a M\"obius transformations [resp. linear transformations] are again a M\"obius transformations [resp. linear transformations].  Hence M\"obius transformation [resp. linear transformation] are continuous in the topology induced by $\cS_G$.
	
	On the other hand every M\"obius transformation [resp. linear transformation] is the inverse function of a M\"obius transformation [resp. linear transformation] and therefore every element of $\cS_G$ is the preimage of $G$ under a M\"obius transformation [resp. linear transformation]. Hence there can be no coarser topology  for which $G$ is open and for which M\"obius transformations [resp. linear transformations] are continuous.
\end{proof}

We will denote  the topology induced by a valuation $v$ by $\cT_v$ and the topology induced by a valuation ring $\cO$ by $\cT_{\cO}$. 
We will examine the relation between $\cT_G$ and  $\cT_{\cO_G}$. 

\begin{cl}\label{clGopeninO}
	Let $v$ be weakly compatible. Then $G$ is open with respect to the topology $\cT_v$. 
\end{cl}

\begin{proof}
	Let $\cA$ be an $\cO_v$-ideal with $\cA\subseteq G$ $\left[\textrm{resp. }1+\cA\subseteq G\right]$ and $\sqrt{\cA}=\cM$. 
	Let $a\in \cA$. Then by Remark~\ref{remfracIdiff} $\cA':=\left.\left\{x\in K\  \right|\   v\left(x\right)>v\left(a\right)\right\}$ is an open subset of $\cA$. 
	
	If $G$ is an additive subgroup of $K$, then for every $x\in G$ as well $x+\cA'$ is open in $\cT_v$.
	As  $x+\cA'\subseteq x+\cA\subseteq x+G\subseteq G$ for all $x\in G$ and $0\in \cA'$, we have $G=\bigcup_{x\in G}\left(x+\cA'\right)$.
	
	If $G$ is a multiplicative subgroup of $K^\times$, $g\cdot \left(1+\cA'\right)\subseteq g\cdot \left(1+\cA\right)\subseteq G$ for all $g\in G$. 
	As $1\in 1+\cA'$ this implies $\displaystyle G=\bigcup_{g\in G} g\cdot\left(1+\cA'\right)$. 
	
	$\cA'$ is open in $\cT_v$ and therefore, as $\cT_v$ is a field topology, $G$ is open.
\end{proof}

\begin{prop}\label{propBasisofO}
	Assume $\cO$ is weakly compatible. 
	\begin{enumerate}[(a)]
		\item Let $G\subseteq K$ be an additive subgroup. 
		Then 
		$\cS_G
		$
		is a basis of $\cT_\cO$.
		\item Let $G\subseteq K^\times$ be a multiplicative subgroup.
		Then \\
		$\left\{\left(a_1\cdot G+b_1\right)\cap \left(a_2\cdot G+b_2\right)\mid a_1,\,a_2\in K^\times,\ b_1,\,b_2\in K\right\}$
		is a basis of $\cT_\cO$.
	\end{enumerate}
\end{prop}

\begin{proof}
	First note that
	$\left\{\alpha\cdot\cM_G+\beta\mid  \alpha\in K^\times, \beta\in K\right\}$
	is a basis of $\cT_{v_G}$.
	\begin{enumerate}[(a)]
		\item Let $a,b,c,d\in K$ such that $a\cdot d-b\cdot c\neq 0$.
		As $G\in \cT_\cO$ by Claim~\ref{clGopeninO}, $G\setminus\left\{-\frac{d}{c}\right\}\in \cT_{\cO}$. As field operations are continuous in $\cT_\cO$ $\left.\left\{\frac{a\cdot x+b}{c\cdot x+d}\,\right|\,x\in G,x\neq-\frac{d}{c}\right\}\in \cT_\cO$. Hence $\cS_G\subseteq \cT_\cO$ and therefore $\cT_G\subseteq \cT_{\cO}$. 
		
		To prove $\cT_{\cO}\subseteq \cT_G$ let $\cA$ be an $\cO_G$-ideal with $\cA\subseteq G$ and $\sqrt{\cA}=\cM_G$. 
		We can choose $d\in K\setminus G$ with $d^{-1}\in \cA$ as follows. Choose $\widetilde{d}\in K\setminus G$. 
		If $0\neq\widetilde{d}^{-1}\in \cA$ set $d:=\widetilde{d}$. 
		If  $0\neq \widetilde{d}^{-1}\notin \cA$, choose $0\neq e\in \cA$. By Lemma~\ref{lemSubgroupsGeneratedByIdeals}~(a)   $0\neq\left(\widetilde{d}-e^{-1}\right){-1}\in \cA$. 
		
		If $e^{-1}\notin G$ set $d:=e^{-1}$. 
		
		If $e^{-1}\in G$, we have $\widetilde{d}-e^{-1} \notin G$. In this case set $d:=\widetilde{d}-e^{-1}$.

		Let $0\neq \widetilde{a},\widetilde{b}\in \cA$. Let $a:=d^{-1}\cdot\widetilde{a},\, b:=a\cdot\widetilde{b}$ and $U:=\left.\left\{\frac{a\cdot x+b}{x+d}\,\right|\, x\in G\right\}$. 
		
		We have $a\cdot d-b=a\cdot\left(d-\widetilde{b}\right)\neq 0$. Hence $a\cdot d-b\neq 0$. 
		Further $x\neq -\frac{d}{1}$ for all $x\in G$. Therefore $U\in\cS_G$.
		Note that  $v_G\left(d\right)<0$ and $v_G\left(\widetilde{b}\right)>0$.
		Let $x\in G$.
		
		Let us first assume  $v_G\left(x\right)<v_G\left(d\right)$.
		Then $v_G\left(x+d\right)=v_G(x)$. Further
		$
		v_G\left(a\cdot x\right)
		<v_G\left(a\right)
		<v_G\left(a\right)+v_G\left(\widetilde{b}\right)
		=v_G\left(b\right)$.
		Therefore $v_G\left(a\cdot x+b\right)=v_G\left(a\right)+v_G\left(x\right)$. Hence 
		$v_G\left(\frac{a\cdot x+b}{x+d}\right)=v_G\left(a\right)>0$.

		Now assume  $v_G\left(x\right)\geq v_G\left(d\right)$. Then 
		$v_G\left(a\cdot x\right)
		\geq v_G\left(a\right)+v_G\left(d\right)
		=v_G\left(\widetilde{a}\right)$.
		
		As $\widetilde{a}\in \cA$ by Remark~\ref{remfracIdiff}~(a) we have $a\cdot x\in \cA$ and therefore $a\cdot x+b\in \cA$.
		As $x+d\notin G$ we have $x+d\notin \cA\subseteq G$. 
		Again by Remark~\ref{remfracIdiff}~(a) follows
		$v_G\left(a\cdot x+b\right)>v_G\left(x+d\right)$ and therefore 
		$v_G\left(\frac{a\cdot x+b}{x+d}\right)>0$. Hence $\frac{a\cdot x+b}{x+d}\in\cM_G$.
		
		That shows $U\subseteq \cM_G$. 
		For $\alpha\in K^\times$ and $\beta\in K$ we have $\alpha\cdot U+\beta\subseteq \alpha\cdot \cM_G+\beta$ and $\alpha\cdot \cU+\beta \in \cS_G$.
		\item Let $n\in \N$, $a_1,\ldots, a_n\in 	K^\times$ and $b_1,\ldots, b_n\in K$. 
		By  Claim~\ref{clGopeninO}  $G\in \cT_\cO$. As field operations are continuous in $\cT_{\cO}$ and $\cT_\cO$ is a topology,
		$\bigcap_{i=1}^n \left(a_i\cdot G+b_i\right)\in \cT_\cO$. Hence $ \cS_G\subseteq \cT_\cO$ and therefore
		$\cT_G\subseteq \cT_\cO$. 
		
		To show $\cT_\cO\subseteq \cT_G$ let $\cA$ be an $\cO_G$-ideal with $1+\cA\subseteq G$ and $\sqrt{\cA}=\cM_G$.
		
		Suppose $c\in K^{\times}$ and $\cA\subseteq c\cdot G\cup\{0\}$. Then for all $0\ne a\in \cA \left\{0\right\}$ there exists  $x\in G$ with $a=c\cdot x$. As $\cA$ is an ideal we have $0\neq \left(c\cdot x\right)^2\in\cA$. Hence $\left(c\cdot x\right)^2\in c\cdot G$ and therefore $c\cdot x^2\in G$. Hence $c\in G$ as  $x^{-2}\in G$. Therefore $ c\cdot G\subseteq G$. Hence $G$ contains all non-zero elements of $\cA$ and hence the group generated by them. But by Lemma~\ref{lemSubgroupsGeneratedByIdeals}~(b) this contradicts $G\neq K^{\times}$.
		
		Therefore there exist $c, d\in K^{\times}$ with $\cA\cap c\cdot G\neq \emptyset$, $ \cA\cap d\cdot G\neq \emptyset$ and $c\cdot G\cap d\cdot G = \emptyset$.
		
		Let $a\in \cA\cap c\cdot G$ and $b\in \cA\cap d\cdot G$.
		Suppose $\left(a-c\cdot G\right)\cap\left(b-d\cdot G\right)\not \subseteq \cM_G$. Let $x\in \left( \left(a-c\cdot G\right)\cap\left(b-d\cdot G\right)\right)\setminus \cM_G$. Then there exist $g_1, g_2\in G$ with
		$x=a-c\cdot g_1=b-d\cdot g_2$.  As  $x^{-1}\in \cO_G$ we have $a\cdot x^{-1}\in \cA$ and $b\cdot x^{-1}\in \cA$. Therefore 
		$
		-c\cdot g_1=x-a=x\cdot\left(1-a\cdot x^{-1}\right)\in x\cdot\left(1+\cA\right)\subseteq x\cdot G
		$
		and
		$
		-d\cdot g_2=x-b=x\cdot\left(1-b\cdot x^{-1}\right)\in x\cdot\left(1+\cA\right)\subseteq x\cdot G.
		$
		Hence there exist $h_1, h_2\in G$ with $-c\cdot g_1=x\cdot h_1$ and $-d\cdot g_2=x\cdot h_2$. We have $g_1\cdot h_1^{-1}\in G$ and $g_2\cdot h_2^{-1}\in G$ and therefore $-x=c\cdot g_1\cdot h_1^{-1}\in c\cdot G$ and $-x=d\cdot g_2\cdot h_2^{-1}\in d\cdot G$. Hence $-x\in c\cdot G\cap d\cdot G$ but this contradicts $c\cdot G\cap d\cdot G=\emptyset$.

		Therefore $U:=\left(-c\cdot G+a\right)\cap\left(-d\cdot G+b\right)\subseteq \cM_G$ and $U\in \cS_G$. 
		
		For $\alpha\in K^\times$ and $\beta\in K$ we have $\alpha\cdot U+\beta\subseteq \alpha\cdot \cM_G+\beta$ and $\alpha\cdot \cU+\beta \in \cS_G$.
	\end{enumerate}
	\vspace{-3ex}
\end{proof}

\begin{lem}\label{lemTopologyInducedbyweaklycompval}
	Let $G \subsetneq K$ $\left[\textrm{resp. }G\subsetneq K^\times\right]$. 
	Then $\cT_{\cO}=\cT_G$ if and only if there exists a non-trivial weakly compatible  coarsening $\cO'$ of $\cO$.
	In this case 
	$\cB_G:=\cS_G$ 
	$\left[\textrm{resp. }\cB_G:=\left.\big\{\left(a\cdot G+b\right)\cap \left(c\cdot G+d\right)\  \right|\   a, c\in K^{\times}, b, d\in K\big\}\right]$ is a basis of $\cT_G$.
\end{lem}

\begin{proof}
	Let us first assume $\cT_{\cO}=\cT_G$. As $G$ is open in $\cT_G=\cT_{\cO}$ and the $\cO$-ideals form a basis of neighbourhoods of  zero of $\cT_{\cO}$, there exists an $\cO$-ideal $\cA\neq\left\{0\right\}$ such that $\cA\subseteq G$ $\left[\textrm{resp. }1+\cA\subseteq G\right]$.
	$\cO':=\cO_{\sqrt{\cA}}\supseteq \cO$ is a valuation ring with maximal ideal $\cM'=\sqrt{\cA}$ and  $\cO\subseteq \cO'$. Hence $\cO'$ is weakly compatible.
	
	Now assume $\cO'\supseteq \cO$ is weakly compatible. By Proposition~\ref{propBasisofO} $\cB_G$
	is basis of $\cT_{\cO'}$ and hence 
	$\cT_{\cO'}=\cT_G$. As $\cO'$ and $\cO$ are dependent  $\cT_{\cO}=\cT_{\cO'}=\cT_G$ (see \cite[Theorem~2.3.4]{EnPr2005}).
\end{proof}

\begin{thm}\label{thmExWeakCompiffVTop}
	Let $K$ be a field with a proper additive subgroup $G$ or with a proper multiplicative subgroup $G$ such that $G\cup\{0\}$ is not an ordering. 
	Then there is a non-trivial  weakly compatible valuation ring if and only if $\cT_G$ is a V-topology.
\end{thm}

\begin{proof}
	Let $\cO$ be a weakly compatible  valuation ring. Then by Lemma~\ref{lemTopologyInducedbyweaklycompval} $\cT_\cO=\cT_G$ and therefore by Theorem~\ref{thmVtopiff} $\cT_G$ is a V-topology.
	
	On the other hand let $\cT_G$ be a V-topology.
	As  $G$ is open with respect to $\cT_G$ by Proposition~\ref{propgroupOpenforarchAbs}  $\cT_G$ cannot be induced by an archimedean absolute value. Hence  by Theorem~\ref{thmVtopiff} $\cT_G$ is induced  by a valuation ring $\cO$.
	By Lemma~\ref{lemTopologyInducedbyweaklycompval} there exists a non-trivial weakly compatible coarsening of $\cO$. 
\end{proof}

\begin{cor}\label{corOGnontrivialiffexiffVtop}
	Let $G\subsetneq K$ be a proper additive subgroup of $K$. [Resp. let $G\subsetneq K^\times$ be a proper multiplicative subgroup of $K$ such that $G\cup\{0\}$ is not an ordering on $K$.] The following are equivalent
	\begin{enumerate}[(i)]
		\item $\cO_G$ is non-trivial. 
		\item There exists a non-trivial weakly compatible valuation ring $\cO$ on $K$.
		\item $\cT_G$ is a V-topology.			
		\item $\cB_G$
		is a basis of a V-to\-po\-lo\-gy.
	\end{enumerate}
\end{cor}	

This follows at once by Lemma~\ref{lemNonTrivialIFF}, Theorem~\ref{thmExWeakCompiffVTop} and  Lemma~\ref{lemTopologyInducedbyweaklycompval}.

\begin{lem}\label{corWeakCompExVTop}
	Let $G\subsetneq K$ be a proper additive subgroup of $K$. [Resp. let $G\subsetneq K^\times$ be a proper multiplicative subgroup of $K$ such that $G\cup\{0\}$ is not an ordering on $K$.] Let $\cL_G:= \left\{ +, -, \cdot\,;0,1; \underline{G}\right\}$, where $G$ is a unary relation symbol. 
	Then any of the equivalent assertions is an elementary property in $\cL_G$.
\end{lem}

\begin{proof}
	We can express in $\cL_G$, that the axioms (V\,1) to (V\,6) hold for $\cB_G$ and hence
	by Remark~\ref{remVtopifBasisV1V6} that $\cB_G$ is a basis of a V-topology. 
\end{proof}

\section{Criteria for the Definability of $\cO_G$}\label{chDefinability}

Let $\cL$ always denote a language and 
$\cL(K)$  the  extension of the language $\cL$ by adding a constant for every element of $K$.
\begin{de}\label{deDefinableValuation}
	\begin{enumerate}[(a)]
		\item We call $\cO$ \emph{$\cL$-definable} (with parameters) or \emph{definable in $\cL$}, if there exists an
		$\cL(K)$-formula $\varphi(x)$ such that $\cO=\left\{x\in K\mid \varphi\left(x\right)\right\}$. We say $\varphi$ defines $\cO$.    
		\item We call  $v$  \emph{$\cL$-definable} if $\cO_v$ is $\cL$-definable.
		\item We call $\cO$ (resp. $v$) \emph{$\cL$-$\emptyset$-definable} or \emph{parameterfree $\cL$-definable}, if the formula $\varphi$ above, is an $\cL$-formula.
		\item We call $\cO$ (respectively $v$) \emph{definable} if it is $\cL_\ring$-definable.
	\end{enumerate}
\end{de}

In some of the theorems in Section~5 we need assumptions that might only be fulfilled in a finite field extension of $K$ but not in $K$ itself. With the following theorem we will still obtain a definable valuation on $K$. 

\begin{prop}\label{propDefValOnSubfields}
	Let $L/K$ be a finite field extension. 
	If $\cO$ is a non-trivial definable valuation ring on $L$, then $\cO\cap K$ is a non-trivial definable valuation ring on $K$. 
\end{prop}

\begin{proof}
	As $L/K$ is algebraic, if $\cO$ is non-trivial, then $\cO\cap K$ is also non-trivial. 
	
	As $L/K$ is finite, $L$ is interpretable in $K$ and hence $\cO\cap K$ is definable. %
\end{proof}

Note that if $\cO$ in the proposition above is parameter-free definable, it does not follow that $\cO\cap K$ is parameter-free definable in $K$.

\begin{ex}\label{exAX}
	For every prime number $q\in \N$ the $q$-adic valuation is definable in the $q$-adic numbers $\Q_q$. The valuation ring is  $\cO_q=\left\{x\in \Q_q\mid \exists y\ y^2-y=q\cdot x^2\right\}$.
\end{ex}

This follows from \cite{Ax1965}.

We now want to explore under which conditions $\cO_G$ is definable in $\cL_G:= \left\{ 0, 1;+, -, \cdot\, ; \underline{G}\right\}$. We will first look at the group case, then at the weak case and at last at the residue case. 

The proofs all follow the same pattern. Let $\cL':=\cL_G(\underline{\cO})$, the language $\cL_G$ extended  by a unary relation symbol. We will show that under certain assumptions for  $(K',G',\cO')\equiv (K,G,\cO_G)$ we have $\cO'=\cO_{G'}$. Hence for every $(K',G')\equiv (K,G)$ there exists at most one~$\cO'$ such that $(K',G',\cO')\equiv (K,G,\cO_G)$ and therefore $\underline{\cO}$ is implicitly defined in $\Th(K,G,\cO_G)$. By Beth's Theorem (see for example \cite[Theorem~9.3]{Po2000}) $\underline{\cO}$ is explicitly defined in $\Th(K,G,\cO_G)$. Hence there exists an $\cL_G$-formula $\varphi$ such that $Th(K,G,\cO_G)\vdash \forall x\ \varphi(x)\leftrightarrow \underline{\cO}(x)$ and hence $\cO$ is $\cL_G$-definable.  

We will further prove that the assumptions for the existence of an $\cL_G$-formula $\varphi$ such that  $\varphi$ defines $\cO_{G'}$ for all $(K',G')\equiv (K,G)$ that we give, are not only sufficient but also necessary. For this we will use the following easy observation. 

\begin{rem}\label{remNotDefinableelemEquiv}
	Let $\cL_G:=\left\{ 0, 1;+, -, \cdot\, ; \underline{G}\right\}$ and $\cL'=\cL_G(\underline{\cO})$. 
	If there exists
	$(K',G',\cO')\equiv (K,G,\cO_G)$ such that $\cO'\neq \cO_{G'}$, then  there exists no 
	$\cL_G$-formula $\varphi$ such that $\varphi$ defines $\cO_{G'}$ for  all $(K',G')\equiv(K,G)$. 
\end{rem}

For the proof of Theorem~\ref{thmGroupDefinableIFF} we will need the following lemma.
\begin{lem}\label{lemexistsn}
	\begin{enumerate}[(a)]
		\item Let $G\subsetneq K$ be an additive subgroup of $K$ such that the group case holds. 
		Let $\cO_G$ be discrete. Let $x_0 \in K$ such that $\cM_G=x_0 \cdot\cO_G$.  
		Then there exists $n\in \mathbb{N}$ such that $x_0 ^{-n}\cdot\cO_G\subseteq G$ and $x_0 ^{-\left(n+1\right)}\cdot\cO_G\nsubseteq G$.
		\item  Let $G\subsetneq K^\times$  be a multiplicative subgroup of $K$ such that the group case holds. 
		Let $x\in \cM_G$.
		Then $\cM_G\setminus x\cdot \cM_G\not\subseteq G$.
	\end{enumerate}
\end{lem}

\begin{proof}
	\begin{enumerate}[(a)]
		\item
		As we are in the group case by Corollary~\ref{corCaseDistinction} $\cO_G\subseteq G$ and therefore  for all $y\in K\setminus G$ we have  $v_G\left(y^{-1}\right)>0$. 
		Assume for all $y\in K\setminus G$ we have $v_G\left(y^{-1}\right)> n$ for all $n\in \mathbb{N}$. 
		Let $\fp:=\left.\left\{z\in K \ \right|\   v_G\left(z\right)> n \textrm{ for all } n\in \mathbb{N}\right\}\neq\emptyset$. By Remark~\ref{remfracIdiff}~(c) $\fp$ is a prime ideal of $\cO_G$ and hence 
		$\cO_\fp:=\left(\cO_G\right)_\fp$
		is a valuation ring on $K$ with $\cO_G\subsetneq \cO_\fp$. 
		Let  $z\in \cO_\fp$. Then there exist $a, b\in \cO_G$ with $b\notin \fp$ and  $z=a\cdot b^{-1}$. As $b\notin \fp$ there exists $n\in \N$ with $v_G\left(b\right)\leq n$. We have
		$v_G\left(z^{-1}\right)=v_G\left(b\right)-v_G\left(a\right)\leq n-v_G\left(a\right) \leq n$. Hence by assumption  $z\in G$. Hence $\cO_G\subsetneq \cO_\fp\subseteq G$. This contradicts Theorem~\ref{thmOG}~(b).
		Choose $y\in K \setminus G$ such that $v_G(y^{-1})>0$ is minimal. Then $v_G(y^{-1})\in \N$. 
		By Lemma~\ref{lemDiscreteValuations} $v_G\left(x_0\right)=1$ and there exists  $a\in \cO_G^{\times}$ such that  $y^{-1}=x_0^{n+1}\cdot a$.  
		Hence $G\not\ni y=x_0 ^{-\left(n+1\right)}\cdot a^{-1}\in x_0 ^{-\left(n+1\right)}\cdot \cO_G$.
		Hence $G\not\supseteq  x_0^{-\left(n+1\right)}\cdot \cO_G$.
		
		Assume $z\in \left(x_0 ^{-n}\cdot\cO_G\right)\setminus G$. Then $z=x_0^{-n}\cdot b$ for some $b\in \cO_G$.
		As $v_G\left(z\right)=v_G\left(x_0 ^{-n}\right)+v_G\left(b\right)\geq -n$ we have
		$v_G\left(z^{-1}\right)\leq n< n+1 =v_G\left(y\right)$.
		But this contradicts the minimality of $v_G\left(y^{-1}\right)$.  
		
		Hence we have found $n\in \N$ with $x_0 ^{-n}\cdot\cO_G\subseteq G$ and $x_0 ^{-\left(n+1\right)}\cdot\cO_G\nsubseteq G$.
		\item Assume there exists $x_0\in \cM_G$ such that 
		$\cM_G\setminus x_0\cdot \cM_G\subseteq G$. 
		Let \\
		$\fp:=\left.\left\{y\in K \ \right|\   v_G\left(y\right)>n\cdot v_G\left(x_0\right)\textrm{ for all } n\in \mathbb{N}\right\}$.
		By Remark~\ref{remfracIdiff}~(c) $\fp$ is a prime ideal of $\cO_G$ and therefore $\left(\cO_G\right)_{\fp}=:\cO_\fp$ is a coarsening of $\cO_G$.
		As  $x_0^{-1}\in \cO_\fp\setminus \cO_G$ we have $\cO_\fp\supsetneq \cO_G$.
		Let $\langle\cM_G\setminus x_0\cdot \cM_G\rangle$ denote the smallest multiplicative subgroup of $K^\times$ which contains $\cM_G\setminus x_0\cdot \cM_G$.
		As $\cM_G\setminus x_0\cdot \cM_G\subseteq G$  we have $\langle\cM_G\setminus x_0\cdot \cM_G\rangle\subseteq G$.
		Let $y\in K$ such that
		$0<v_G\left(y\right)\leq m\cdot v_G\left(x_0\right)$ and  $\left(m-1\right)\cdot v_G\left(x_0\right)<v_G\left(y\right)$ for some $m\in \mathbb{N}$.
		Then 
		$y\cdot x_0^{-(m-1)}\in \cM_G$. Further
		$
		v_G\left(y\cdot x_0^{-(m-1)}\right)\leq m \cdot v_G\left(x_0\right)-\left(m-1\right)\cdot v_G\left(x_0\right)
		=v_G\left(x_0\right)$. 
		Thus
		$y\cdot x_0^{-(m-1)}\in \cM_G\setminus x_0\cdot \cM_G$. 
		As $x_0\in \cM_G\setminus x_0\cdot \cM_G$ and therefore
		$x_0^{m-1}\in \langle\cM_G\setminus x_0\cdot \cM_G\rangle $ it follows that $y=x_0^{m-1}\cdot y\cdot x_0^{-(m-1)}\in \langle\cM_G\setminus x_0\cdot \cM_G\rangle$.  
		
		Now let $y\in \cO_\fp^{\times}\setminus \cO_G^\times$. Then  $y\notin \fp$ and $y^{-1}\notin \fp$.  Hence there exist $n_1, n_2 \in\mathbb{N}$ such that $v_G\left(y\right)\leq n_1\cdot  v_G\left(x_0\right)$ and $v_G\left(y^{-1}\right)\leq n_2\cdot v_G\left(x_0\right)$. 
		As $y\notin \cO_G^\times$ by assumption, we have $v_G(y)\neq 0$. 
		If $v_G\left(y\right)>0$ then  $y\in  \langle \cM_G\setminus x_0\cdot \cM_G\rangle$ as shown above.
		If $v_G\left(y\right)<0$ then $y^{-1}\in  \langle \cM_G\setminus x_0\cdot \cM_G\rangle$ and hence $y \in \langle \cM_G\setminus x_0\cdot \cM_G\rangle$. 
		Therefore $\cO_\fp^\times\setminus \cO_G^\times\subseteq \langle \cM_G\setminus x_0\cdot \cM_G\rangle$. 
		As $\cO_G^\times\subseteq G$ and  $\langle \cM_G\setminus x_0\cdot \cM_G\rangle\subseteq G$ we have $\cO_\fp^\times\subseteq G$. But this  contradicts Theorem~\ref{thmOG}~(b).
	\end{enumerate}
\end{proof}

\begin{thm}\label{thmGroupDefinableIFF}
	\begin{enumerate}[(a)]
		\item Let $G$ be an additive subgroup of $K$ such that the group case holds.  Then there exists an $\cL_G$-formula $\varphi$ such that  $\varphi$ defines $\cO_{G'}$ for all  $(K',G')\equiv (K,G)$ if and only if $\cO_G$ is discrete or  $x^{-1}\cdot\cO_G\not\subseteq G$ for all $x\in \cM_G$.
		\item Let $G\subsetneq K^\times$ be a multiplicative subgroup of  $K$ such that the group case holds. 
		Then there exists an $\cL_G$-formula $\varphi$ such that  $\varphi$ defines $\cO_{G'}$ for all $(K',G')\equiv (K,G)$.
	\end{enumerate} 
\end{thm}

\begin{proof}
	\begin{enumerate}[(a)]
		\item Let $(K',G',\cO')\equiv \left(K,G,\cO_G\right)$ be an $\cL'$-structure. 
		Let $\cM'$ denote the maximal ideal of $\cO'$.
		As  $\cO_G\subseteq G$, we have $\cO'\subseteq G'$. 
		Hence we are in the group case and therefore by Corollary~\ref{corCaseDistinction} we have $\cO_{G'}\subseteq G'$ and $\cO'\subseteq \cO_{G'}$. 
		
		Let us first assume that $\cO_G$ is discrete.
		By Lemma~\ref{lemDiscreteValuations} there exists $x_0\in K$ such that $\cM_G=x_0\cdot \cO_G$. By Lemma~\ref{lemexistsn}~(a) there exists $n\in \N$ such that $x_0 ^{-n}\cdot\cO_G\subseteq G$ and $x_0 ^{-\left(n+1\right)}\cdot\cO_G\nsubseteq G$. 
		As $(K',G',\cO')\equiv \left(K,G,\cO_G\right)$ there exists $x'\in K'$ such that $x'\cdot \cO'=\cM'$, $\left(x'\right)^{-n}\cdot \cO'\subseteq G'$ and $\left(x' \right)^{-(n+1)}\not\subseteq G'$. 
		Assume $x'\notin \cM_{G'}$. Then $\left(x'\right)^{-1}\in \cO_{G'}$ and thus $\left(x'\right)^{-(n+1)}\cdot \cO'\subseteq \left(x'\right)^{-(n+1)}\cdot \cO_{G'}\subseteq \cO_{G'}\subseteq G'$. But this contradicts the choice of $x'$. 
		Hence $x'\in \cM_{G'}$ and therefore $x'\cdot \cO_{G'}\subseteq \cM_{G'}$. Thus $\cM'=x'\cdot \cO'\subseteq x'\cdot \cO_{G'}\subseteq \cM_{G'}$ and therefore $\cO_{G'}\subseteq \cO'$.
		Altogether follows $\cO_{G'}=\cO'$.

		Now assume $x^{-1}\cdot\cO_G\not\subseteq G$ for all $x\in \cM_G$. 
		Assume $\cO'\subsetneq \cO_{G'}$. Let $x\in \cM'\setminus\cM_{G'} $. Then $x^{-1}\in \cO_{G'}^\times$ and therefore 
		$x^{-1}\cdot \cO'\subseteq x^{-1}\cdot \cO_{G'}= \cO_{G'}\subseteq G'$. But as $(K',G',\cO')\equiv (K,G,\cO_G)$ this is a contradiction. Therefore $\cO'=\cO_{G'}$. 
		

		Hence in both cases by Beth's Theorem there exists an $\cL_G$-formula $\varphi$ such that  $\varphi$ defines $\cO_{G'}$ for all $(K',G')\equiv (K,G)$.
		
		Finally assume $x\in \cM_G$ such that $x^{-1}\cdot\cO_G\subseteq G$ and $\cO_G$ is not discrete.
		Then for every $n\in \mathbb{N}$ there exists $y_n\in \cM_G\setminus \left\{0\right\}$ such that 
		$v_G\left(x\right)\geq n\cdot v_G\left(y_n\right)\geq k\cdot v_G\left(y_n\right)$ for all $k\leq n$.
		For all $a\in \cO_G$ we have  $x\cdot a\cdot y_n^{-k}\in \cO_G$ and therefore $y_n^{-k}\cdot a\in x^{-1}\cdot\cO_G$. Thus $y_n^{-k}\cdot\cO_G\subseteq x^{-1}\cdot \cO_G\subseteq G$ for all $k\leq n$.
		Hence 
		$\Phi\left(y\right)=\left.\left\{ y\in \cM_G\wedge  0\neq y \wedge y^{-n}\cdot \cO_G\subseteq G \ \right|\   n\in \mathbb{N}\right\}
		$
		is a finitely satisfiable type. Thus
		there exists an elementary extension $\left(K', G', \cO'\right)$ of $(K,G,\cO_G)$ and $y'\in K'$ such that $y'$ realizes $\Phi\left(y\right)$.
		Let $ \cO''=\bigcup_{n=0}^{\infty}\left(y'\right)^{-n}\cdot\cO'$.
		As $\left(y'\right)^{-n}\cdot\cO'\subseteq G'$ for every $n\in \mathbb{N}$, we have $\cO''\subseteq G'$ . 
		Further $\cO'\subseteq \cO''$.
		As $y'\in \cM'$ we have $ \left(y'\right)^{-1}\notin \cO'$ but $\left(y'\right)^{-1}\in \left(y'\right)^{-1}\cdot\cO'\subseteq \cO''$ and therefore $\cO'\subsetneq \cO''\subseteq G'$. 
		Thus $\cO'\neq \cO_{G'}$. 
		Hence by Remark~\ref{remNotDefinableelemEquiv} there exists no $\cL_G$-formula $\varphi$ such that $\varphi$ defines $\cO_{G'}$ for all $(K',G')\equiv (K,G)$. 
		\item Let $(K',G',\cO')\equiv (K,G,\cO_G)$. As $\cO_G^\times \subseteq G$ we have $\left(\cO'\right)^\times\subseteq G'$. By Corollary~\ref{corCaseDistinction} we have $\cO'\subseteq \cO_{G'}$ and $\cO_{G'}^\times\subseteq G'$. 
		Assume $\cO'\subsetneq \cO_{G'}$. Let $x\in \cM'\setminus \cM_{G'}$.
		As $x\in \cO_{G'}^{\times}$ we have $x\cdot \cM_{G'}=\cM_{G'}$. 
		Therefore 
		$
		\cM'\setminus x\cdot \cM'\subseteq \cM'\setminus x\cdot \cM_{G'} = \cM'\setminus \cM_{G'}\subseteq \cO_{G'}^{\times}\subseteq G'$.
		Hence there exists $x\in \cM'$ such that  $\cM'\setminus x\cdot \cM'\subseteq G'$. But as by Lemma~\ref{lemexistsn}~(b)  $\cM_G\setminus x\cdot \cM_G\not\subseteq G$, this contradicts $(K',G',\cO')\equiv (K,G,\cO_G)$.
		
		Therefore $\cO'=\cO_{G'}$ and hence by Beth's Theorem  there exists an $\cL_G$-formula $\varphi$ such that $\varphi$ defines $\cO_{G'}$ for all $(K',G')\equiv (K,G)$.
	\end{enumerate}
\end{proof}

\begin{thm}\label{thmWeakDefinableIFF}
	Let $G\subsetneq K$ $[\textrm{resp. }G\subsetneq K^\times]$ be a subgroup of  $K$ such that the weak case holds. 
	Then there exists an $\cL_G$-formula $\varphi$ such that  $\varphi$ defines $\cO_{G'}$ for all
	$(K',G')\equiv (K,G)$ if and only if $\cO_G$ is discrete.
\end{thm}

\begin{proof}
	Let us first assume that $\cO_G$ is discrete. 
	Let $\cA$ be an $\cO$-ideal with
	$\cA\subseteq G$ $\left[\textrm{resp. }1+\cA\subseteq G\right]$ and $\cM_G=\sqrt{\cA}$. 
	Let $x_0\in \cM_G$ with $v_G\left(x_0\right)=1$. Let $a\in \cA$ and  $k\in \mathbb{N}$ such that $x_0^k=a$. Then $v_G\left(a\right)=k\in \N$.
	Choose $y_0\in \cM_G\setminus G$ $\left[\textrm{resp. }y_0\in \cM_G\setminus G-1\right]$ such that $v_G\left(y_0\right)$ is maximal. Such a $y_0$ exists as $v_G\left(\cM_G\setminus G\right)$ $\left[\textrm{resp. }v_G\left(\cM_G\setminus G-1\right)\right]$ is bounded by $v_G\left(a\right)$ by Remark~\ref{remfracIdiff} and $v_G$ is discrete. 
	As $y_0\notin G\supseteq \cA$ $\left[\textrm{resp. }y_0\notin G-1\supseteq \cA\right]$  by Remark~\ref{remfracIdiff} we have $0<v_G\left(y_0\right)<v_G\left(a\right)=k$. Hence $v_G\left(y_0\right)\in \mathbb{N}$.
	From Lemma~\ref{lemDiscreteValuations} follows that there exists $n \in \mathbb{N}$ and $b\in \cO_G^{\times}$ such that  $y_0=x_0^n\cdot b$.  
	Hence $y_0\in x_0^n\cdot \cO_G\setminus G$ $\left[\textrm{resp. }y_0\in x_0^n\cdot \cO_G\setminus (G-1)\right]$ and therefore $G\not\supseteq  x_0^{n}\cdot \cO_G$  $\left[\textrm{resp. }G-1\not\supseteq  x_0^{n}\cdot \cO_G\right]$.
	Assume there exists $\displaystyle z\in x_0^{n+1}\cdot\cO_G\setminus G$ $\displaystyle \left[\textrm{resp. }z\in x_0^{n+1}\cdot\cO_G\setminus G-1\right]$.
	Let $z_0\in \cO_G$ such that $z=x_0^{n+1}\cdot z_0$.
	We have $v_G\left(z\right)\geq n+1>v_G\left(y_0\right)$.
	But this contradicts the maximality of $v_G\left(y_0\right)$.  Hence $x_0^{n+1}\cdot\cO_G\subseteq G$ $\left[\textrm{resp. }1+x_0^{n+1}\cdot\cO_G\subseteq G\right]$.

	Now let $(K',G',\cO')\equiv (K,G,\cO_G)$. 
	Let $\cM'$ be the maximal ideal of $\cO'$. As $\cO_G$ is not compatible with $G$,  $\cO'$ is not compatible with $G'$. 
	Further there exists $x'\in K'$ such that $x' \cdot\cO'=\cM'$, $\left(x'\right)^n\cdot\cO'\not\subseteq G'$ and  $\left(x'\right)^{n+1}\cdot\cO'\subseteq G'$ $\Big[\textrm{resp. }x' \cdot\cO'=\cM'$, $1+\left(x'\right)^n\cdot\cO'\not\subseteq G'$ and $1+\left(x'\right)^{n+1}\cdot\cO'\subseteq G'\Big]$.
	Let $v'$ be a valuation with $\cO_{v'}=\cO'$. 
	Let $\cA:=\left\{a\in K \ \left|\   v'\left(a\right)>v'\left(\left(x'\right)^{n+1}\right)\right.\right\}$.
	$\cA$ is an $\cO'$-ideal with $\cA\subseteq \left(x'\right)^{n+1}\cdot\cO'\subseteq G'$ $\left[\textrm{resp. }1+\cA\subseteq 1+ \left(x'\right)^{n+1}\cdot\cO'\subseteq G'\right]$. Further   for every $z\in \cM'$ there exists $a\in \cO'$ such that
	$z=x'\cdot a$. We have $v'\left(z^{n+2}\right)=v'\left(x'\right)^{n+1}+v'\left(x'\right)+v'\left(a^{n+2}\right)>v'\left(\left(x'\right)^{n+1}\right)$ and hence $z\in \sqrt{\cA}$. Therefore $\sqrt{\cA}=\cM$ and thus $\cO'$ is weakly compatible with $G'$. 
	By Corollary~\ref{corCaseDistinction}   $\cO'=\cO_{G'}$. 
	Hence by Beth's Theorem  if $\cO_G$ is discrete there exists an $\cL_G$-formula $\varphi$ such that  $\varphi$ defines $\cO_{G'}$ for all $(K',G')\equiv (K,G)$.

	Now assume $\cO_G$ is not discrete.
	Let $x_0\in \cM_G\setminus G$ $\left[\textrm{resp. }x_0\in \cM_G\setminus G-1\right]$. Then $x_0 \cdot\cO_G\nsubseteq G$ $\left[\textrm{resp. }x_0 \cdot\cO_G\nsubseteq G-1\right]$.
	As $\cO_G$ is not discrete, for every $n\in \mathbb{N}$ there exists $y\in \cM_G\setminus \left\{0\right\}$ such that $v_G\left(x_0\right)\geq n\cdot v_G\left(y\right)\geq k\cdot v_G\left(y\right)$ for all $k\leq n$.
	For $a\in \cO_G$ we have  $x_0\cdot a\cdot y^{-k}\in \cO_G$. Therefore $x_0\cdot a\in y^k\cdot \cO_G$. Hence
	$y^k\cdot\cO_G\supseteq x_0 \cdot\cO_G\not\subseteq G$ $\left[\textrm{resp. }y^k\cdot\cO_G\supseteq x_0 \cdot\cO_G\not\subseteq G-1\right]$ for all $k\geq n$. 
	Let
	$z\in y^n\cdot\cO_G\setminus G$ $\left[\textrm{resp. }z\in y^n\cdot\cO_G\setminus\left( G-1\right)\right]$.
	As $y\in \cO_G$ we have $y^n\cdot \cO_G\subseteq y^k\cdot\cO_G$ and therefore $z\in y^k\cdot\cO_G$ for every $k\leq n$. Thus there exists $z \in \bigcap_{k=1}^n y^k\cdot\cO_G=y^n\cdot\cO_G$ with $z\notin G$.
	Therefore $
	\Phi\left(y,z\right)=\left\{y\in \cM_G\wedge 0\neq y \wedge z\in y^{n}\cdot\cO_G\wedge z\notin  G \ \left|\   n\in \mathbb{N}\right.\right\}$ 
	$\left[\textrm{resp. }\Phi\left(y,z\right)=\left\{ y\in \cM_G\wedge 0\neq y\wedge z\in y^{n}\cdot\cO_G\wedge z\notin  G-1 \ \left|\   n\in \mathbb{N}\right.\right\}\right]$
	is a fi\-nite\-ly satisfiable type. Hence  
	there exist an elementary extension $\left(K', G', \cO'\right)$ and $y',\, z'\in K'$ such that $\left(y', z'\right)$ realizes $\Phi\left(y,z\right)$.
	Let $\fp=\bigcap_{n=1}^{\infty}(y')^{n}\cdot\cO'$.
	Let $a,b\in \fp$. Then for all $n\in \N$ there exist $a_n,\,b_n\in \cO'$ such that $a=(y')^{n}\cdot a_n$ and  $b=(y')^{n}\cdot b_n$. We have $a+b=(y')^n\cdot(a_n+b_n)\in (y')^n\cdot \cO'$. Hence $a+b\in \fp$.
	Let $c\in \cO'$. For every $n\in \N$ we have $c\cdot a=c\cdot (y')^{n}\cdot a_n\in (y')^{n}\cdot \cO'$. Hence $c\cdot a\in \fp$. 
	Now let $a,b\in \cO'$ with $a\cdot b\in \fp$. Assume $a\notin \fp$. Then there exists $n_0\in \N$ such that $a\notin (y')^{n_0}\cdot \cO'$. Hence $ v_G\left(a\cdot\left(y'\right)^{-n_0}\right)<0$. Let $m\in \N$. We have $a\cdot b\in (y')^{n_0+m}\cdot \cO'$ and thus $ 0\leq  v_G\left(a\cdot(y')^{-n_0}\right)+v_G\left(b\cdot(y')^{-m}\right)$. Hence  we have $ v_G\left(b\cdot(y')^{-m}\right)>0$ and therefore $b\in (y')^m\cdot \cO'$. Thus $b\in \fp$. Hence $\fp$ is an $\cO'$-prime ideal.  
	As $z'\in \fp$ we have $\fp\not\subseteq G'$ $\left[\textrm{resp. }\fp\not\subseteq G'-1\right]$.
	As  $(y')^{n}\cdot\cO'\subseteq \cM'$ for all $n\in \mathbb{N}$ we have $\fp\subseteq \cM'$. 
	As $(y')^{-1}\notin \cO'$ we have $y'\notin (y')^2\cdot\cO'$. Hence $\fp\subsetneq \cM'$.
	By Remark~\ref{remfracIdiff}  for every ideal $\cA\subseteq G'$ $\left[\textrm{resp. }\cA\subseteq G'-1\right]$ we have $\cA\subseteq \fp$ and therefore a $\sqrt{\cA}\subseteq \fp\subsetneq\cM'$.
	Hence $\cO'$ is not coarsely compatible with $G'$. In particular $\cO'\neq \cO_{G'}$. By Remark~\ref{remNotDefinableelemEquiv} there exists no $\cL_G$-formula $\varphi$ such that  $\varphi$ defines $\cO_{G'}$ for all $(K',G')\equiv (K,G)$. 
\end{proof}

For every   subgroup $G$ of $K$, $\overline{G}:=\varrho(G)$ is a subgroup of the residue field $\overline{K}$. We will show the following lemma.

\begin{lem}\label{lemOverlineGandG}
	Let $G$ be a subgroup of $K$ such that the group case or the residue case holds. 
	\begin{enumerate}[(a)]
		\item Let $G\subseteq K$ be an additive subgroup of $K$. Let $x\in K$.
		Then $\overline{x}\in \overline{G}$ if and only if $x\in G$.
		\item Let $G\subseteq K$ be a multiplicative subgroup of $K$. Let $x\in \cO_G^\times$.
		Then $\overline{x}\in \overline{G}$ if and only if $x\in G$.
	\end{enumerate}
\end{lem}

\begin{proof}
	\begin{enumerate}[(a)]
		\item  
		Let $\overline{x}\in \overline{G}$. Then there exists $y\in G$ with $\overline{y}=\overline{x}$ hence $x=y+\alpha$ for some $\alpha \in \cM_G\subseteq G$. As $\alpha, y\in G$, we have $x=y+\alpha \in G$.
		
		The other direction is clear.		
		\item Let $x\in \cO_G^{\times}$.
		Assume $\overline{x}\in \overline{G}$. Then there exists $y\in G$ with $\overline{y}=\overline{x}$ hence $x=y+\alpha$ for some $\alpha \in \cM_G$.
		Let $v_G$ be a valuation with
		$\cO_G=\cO_{v_G}$. We have $v_G\left(y\right)=\min\{v_G(x), v_G(\alpha)\}=0$ and therefore $y\in \cO_G^\times$. Hence $y{-1}\in \cO_G$ and therefore $\alpha\cdot y^{-1}\in \cM_G$. As  $1+\cM_G\subseteq G$  $1+\alpha\cdot y^{-1},  y\in G$. Therefore $x=y\cdot\left(1+\alpha\cdot y^{-1}\right)\in G$.

		The other direction is again clear. 
	\end{enumerate}
	\vspace{-4ex}
\end{proof}

\begin{thm}\label{thmResiduedefinable}
	Let $G\subseteq K$ be a subgroup of a field such that the residue case holds. 
	Then there exists an $\cL_G$-formula $\varphi$ such that  $\varphi$ defines $\cO_{G'}$ for all $(K',G')\equiv (K,G)$  if and only if $G$ is additive or $G$ is multiplicative and $\overline{G}\cup\{0\}$ is no ordering.
\end{thm}

\begin{proof}
	Let us first assume $G$ is additive or $G$ is multiplicative and  $\overline{G}\cup\{0\}$ is no ordering.
	Assume $\cO^*$ is a non-trivial valuation ring on $\overline{K}$ which is weakly compatible with $\overline{G}$.
	Let $\widetilde{\cO}:=\varrho^{-1}\left(\cO^*\right)$. As $\cO^*$ is non-trivial, $\widetilde{\cO}$ is a valuation ring on $K$ with $\widetilde{\cO}\subsetneq \cO_G$. 
	Let $\cM^*$ denote the maximal ideal of $\cO^*$ and $\widetilde{\cM}$ the maximal ideal of $\widetilde{\cO}$. 
	Let $\cA$ be an $\cO^*$-ideal such that $\sqrt{\cA}=\cM^*$ and $\cA\subseteq \overline{G}$ $\left[\textrm{resp. }1+\cA\subseteq \overline{G}\right]$. Then  $\varrho^{-1}\left(\cA\right)$ is an $\widetilde{\cO}$-ideal with $\sqrt{\varrho^{-1}\left(\cA\right)}=\widetilde{\cM}$. With Lemma~\ref{lemOverlineGandG}
	$
	\varrho^{-1}\left(\cA\right)\subseteq\varrho^{-1}\left(\overline{G}\right)
	= G$
	$\left[\textrm{resp. }
	1+\varrho^{-1}\left(\cA\right)\subseteq  \varrho^{-1}\left(1\right)+\varrho^{-1}\left(\cA\right)
	=\varrho^{-1}\left(1+\cA\right)
	\subseteq\varrho^{-1}\left(\overline{G}\right)
	= G\right]$.
	Therefore $\widetilde{\cO}$ is a  weakly compatible refinement of $\cO_G$. As we are in the residue case by Corollary~\ref{corCaseDistinction} this is a contradiction. 
	Hence there exists no non-trivial valuation ring on $\overline{K}$ which is weakly compatible with $\overline{G}$.

	Now let $(K',G',\cO')\equiv(K,G,\cO_G)$.   $\cO'$ is coarsely compatible with $G'$ and hence  $\cO_{G'}\subseteq \cO'$. Assume $\cO_{G'}\subsetneq \cO'$. Let $\varrho'$ denote the residue homomorphism
	$\varrho':\cO'\longrightarrow \cO'/\cM'$. Then $\varrho'(\cO_{G'})$ is a non-trivial valuation ring on $\overline{K'}:=\cO'/\cM'$.
	We have $\varrho'\left(\cM_{G'}\right)\subseteq \varrho'\left(G'\right)=\overline{G'}$
	$\left[\textrm{resp. }1+\varrho'\left(\cM_{G'}\right)=\varrho'\left(1+\cM_{G'}\right)\subseteq \varrho'\left(G'\right)=\overline{G'}\right]$. Therefore $\overline{\cO}_{G'}$ is a non-trivial valuation ring on $\overline{K'}$ which is weakly compatible with $\overline{G'}$. But  this contradicts $(K',G',\cO')\equiv (K,G,\cO_G)$ by Corollary~\ref{corWeakCompExVTop}.

	Now assume $G$ is a multiplicative subgroup of $K^\times$ and $\overline{G}\cup\{0\}$ an ordering on the residue field $\overline{K}$ of $(K,\,\cO_G)$.
	Assume $\overline{G}\cup\{0\}$ is not archimedean. Then the valuation ring 
	$\cO^*:=\left.\left\{x\in \overline{K} \ \right|\   \textrm{there exists } a\in \mathbb{Z} \ a-x\in \overline{G},\, a+x\in \overline{G}\right\}$
	on $\overline{K}$ is non-trivial (compare \cite[page\,36]{EnPr2005}).
	Let $\varrho:\cO_G\longrightarrow \overline{K}$ denote the residue homomorphism.
	Then $\varrho^{-1}\left(\cO^*\right):=\widetilde{\cO}$ is a valuation ring on $K$ with $\cO_G\supsetneq \widetilde{\cO}$. Denote by $\cM^*$ the maximal ideal of $\cO^*$ and by $\widetilde{\cM}$ the maximal ideal of $\widetilde{\cO}$. 
	
	$\cO^*$ is $\left(\overline{G}\cup\{0\}\right)$-convex. Hence $1+\cM^*\subseteq \overline{G}$ (see for example \cite[Proposition~2.2.4]{EnPr2005}). As by Lemma~\ref{lemOverlineGandG}
	$\varrho^{-1}\left(\overline{G}\right)
	= G$ $1+\widetilde{\cM}\subseteq
	\varrho^{-1}\left(1\right)+\varrho^{-1}\left(\cM^*\right)
	=\varrho^{-1}\left(1+\cM^*\right)
	\subseteq\varrho^{-1}\left(\overline{G}\right)
	= G$.
	Hence
	$\widetilde{\cO}\subsetneq \cO_G$ is a coarsely compatible valuation ring on $K$. This is a contradiction. Therefore $\overline{G}\cup\{0\}$ must be an archimedean ordering.
	Let $\Phi\left(y\right):=\left.\left\{y\in \overline{K}\wedge n-y\notin \overline{G} \ \right|\   n\in \mathbb{N}\right\}$.
	For every $n\in \mathbb{N}$ there exists $y\in K$ such that $n-y\notin \overline{G}$  and therefore $k-y\notin \overline{G}$ for all $k\leq n$.
	Therefore  $\Phi\left(y\right)$ is a finitely satisfiable type. Hence 
	there exists an elementary extension $\left(K', G', \cO'\right)$ and  $y'\in K'$ such that $y'$ realizes $\Phi\left(y\right)$.   
	$\overline{G'}\cup\{0\}$ is a non-archimedean order on $\overline{K'}$ as $y'>n$ for all $n\in \mathbb{N}$.
	As above from $\overline{G'}\cup\{0\}$ non-archimedean follows that there exists a valuation ring $\widetilde{\cO}\subsetneq \cO'$ which is compatible with $G'$. As we have $\cO_G^\times\not\subseteq G$ we have  $\left(\cO'\right)^\times\not\subseteq G'$. Hence $\cO'$ has a proper refinement which is coarsely compatible with $G'$ and hence  $\cO'\neq \cO_{G'}$.   By Remark~\ref{remNotDefinableelemEquiv} there exists no $\cL_G$-formula $\varphi$ such that  $\varphi$ defines $\cO_{G'}$ for all $(K',G')\equiv (K,G)$. 
\end{proof}


The following table summarizes Theorem~\ref{thmGroupDefinableIFF},  Theorem~\ref{thmWeakDefinableIFF} and Theorem~\ref{thmResiduedefinable}.
\begin{thm}\label{thmOTDefinable}
	Let  $G\subsetneq K$ $\left[\textrm{resp. }G\subsetneq K^\times\right]$ be a subgroup of $K$.
	
	Then there exists an $\cL_G$-formula $\varphi$ such that  $\varphi$ defines $\cO_{G'}$ for all $(K',G')\equiv (K,G)$  if and only if 
	
	\upshape
	\begin{tabular}{|c||c|c|}\hline
		& \textbf{$G\subseteq K$} \textbf{additive} & \textbf{$G\subseteq K^\times$ multiplicative} \\ \hline\hline
		\textbf{group case} & iff either $\cO_G$ is discrete  & always \\ 
		&or for all $x\in \cM_G$  $x^{-1}\cdot\cO_G\subseteq G$&\\ \hline
		\textbf{weak case}  & \multicolumn{2}{|c|}{if and only if $\cO_G$ is discrete}\\\hline
		\textbf{residue case}& always & iff $\displaystyle \overline{G}\cup\{0\}$ is no ordering\\ \hline
	\end{tabular}
\end{thm}

\section[$\cO_G$ for $G=(K^\times)^q$ and $G=K^{(p)}$]{$\cO_G$ for Groups of Prime Powers and the Artin Schreier Group}
\label{secDefinabilityASandPowers}

In this section we want to apply the results from the previous sections to the Artin-Schreier group $G=K^{(p)}$ for $p=\ch(K)>0$ and the group of prime powers
$G=(K^\times)^q$ for $q\neq \ch(K)$ prime. As these groups are $\cL_{\ring}$-$\emptyset$-definable, any $\cL_G$-$\emptyset$-definable valuation will be $\cL_{\ring}$-$\emptyset$-definable. 

We will start with a lemma that shows, that for these goups the weak case can only occur if $G=(K^\times)^q$ for $q=\ch\left(\overline{K}\right)$.

\begin{lem}\label{lemWeaklyCompIffComp}
	Let $\cO$ be a valuation ring on a field $K$.
	\begin{itemize}
		\item Let $G$ be an additive subgroup of $K$ and $K^{\left(p\right)}\subseteq G$ for $p:=char\left(K\right)>0$ \textbf{ or }
		\item let $G$ be a multiplicative subgroup such that there exists $n\in \mathbb{N}$ with $\left(K^{\times}\right)^n\subseteq G$ and $\gcdi\big(n, char\left(\overline{K}\right)\big)=1$ if $\ch\left(\overline{K}\right)\neq	0$. 
	\end{itemize}
	Then $v$ is compatible if and only if it is weakly compatible.
\end{lem}

\begin{proof}
	It is clear that if $\cO$ is compatible, then it is weakly compatible.
	
	Assume $\cO$ is weakly compatible but not compatible. Let $\cA$ be an $\cO$-ideal  with
	$\sqrt{\cA}=\cM$ and $\cA\subseteq G$ $\left[\textrm{resp. }1+\cA\subseteq G\right]$. By Remark~\ref{remfracIdiff} we can choose $\cA$ maximal with $\cA\subseteq G$ $\left[\textrm{resp. }1+\cA\subseteq G\right]$. 
	Let $a\in \cM\setminus \cA$. Let $k\in \mathbb{N}$ with $a^k\notin \cA$ and $a^{k+1}\in \cA$. 
	Define the $\cO$-ideal $\cB:=a^k\cdot \cO$.
	As $a^k\in \cB\setminus \cA$ we have $\cB\not\subseteq \cA$ and hence by Remark~\ref{remfracIdiff}  $\cA\subsetneq \cB$. 
	Let $x\in \cB^2$. Then there exists $y\in \cO$ with $x=\left(a^k\cdot y\right)^2$. As  $a^{k-1}\cdot y^2\in \cO$ and $a^{k+1}\in \cA$, we have $x\in \cA$. Hence $\cB^2\subseteq \cA$.
	
	Let us first show that if $G$ is an additive subgroup of $K$ then $\cB\subseteq G$.
	Let $b\in \cB$. As $p=\ch(K)\geq 2$, we have $b^{p-2}\in\cO$. Further  $b^2\in \cA$. Therefore $b^p\in \cA$.  
	As $\left(-b\right)^p+b\in K^{(p)}\subseteq G$, therefore $\left(-b\right)^p+b\pm b^{p}\in G$. Thus $\cB\subseteq G$ . 
	
	Now assume that $G$ is a multiplicative subgroup. We will show $1+\cB\subseteq G$.
	
	Let $b\in \cB$.
	Then  
	\[
	G\ni\left(\frac{b}{n}+1\right)^n
	= 1+b+ \binom{n}{2}\cdot \left(\frac{1}{n}\right)^2\cdot  b^2+ b\cdot\left(\sum_{i=0}^{n-3}\binom{n}{i+3}\cdot  \left(\frac{1}{n}\right)^{i+3}\cdot  b^i\right)\cdot  b^2.
	\]
	As $\gcdi\big(n, char\left(\overline{K}\right)\big)=1$ we have $n\in \cO^{\times}$.  Furthermore $b\in \cO$ and for all  $i,\,j \in  \mathbb{N}$ with  $i\leq j$ we have $\binom{i}{j}\in \mathbb{N}\subseteq \cO$. Hence $ \sum_{i=0}^{n-3}\binom{n}{i+3}\cdot  \left(\frac{1}{n}\right)^{i+3}\cdot  b^i\in\cO$ and $\binom{n}{2}\cdot \left(\frac{1}{n}\right)^2\in \cO$.
	As $\cB^2$ is an $\cO$-ideal, from this follows
	$ \binom{n}{2}\cdot \left(\frac{1}{n}\right)^2\cdot  b^2\in \cB^2\subseteq \cA$ and 
	$\left(\sum_{i=0}^{n-3}\binom{n}{i+3} \cdot \left(\frac{1}{n}\right)^{i+3} b^i\right) \cdot b^2\in \cB^2\subseteq \cA$.
	Therefore
	$\left(\frac{b}{n}+1\right)^n\in  1+b+\cA+b\cdot \cA =\left(1+b\right)\cdot \left(1+\cA\right)$. By Lemma~\ref{lemInverseOf1plusIdeal} $\left(1+\cA\right)^{-1}
	=1+\cA$.
	Hence $1+b \in \left(K^\times\right)^q\cdot \left(1+\cA\right)^{-1}
	\subseteq G\cdot \left(1+\cA\right)^{-1}
	=G\cdot  \left(1+\cA\right)
	\subseteq  G$. 
	Hence $1+\cB\subseteq G$.
	
	Therefore $\cB$ is an $\cO$-ideal with $\cB\subseteq G$ $\left[\textrm{resp. }1+\cB\subseteq G\right]$ and $\cA\subsetneq \cB$. But this contradicts the choice of $\cA$. 
\end{proof}

\begin{thm}\label{thmArtinschreierindBewDef}
	Let $K$ be a field with $\textrm{char}\left(K\right)=p>0$ and $G:=K^{\left(p\right)}$.
	Then $\cO_G$ is $\emptyset$-definable. 
\end{thm}

\begin{proof}
	As the case $\cO_G=K$ is trivial we can assume $\cO_G\neq K$ and hence as well $G\neq K$.

	By Lemma~\ref{lemWeaklyCompIffComp} we are not in the weak case.
	
	If we are in the residue case $\cO_G$ is $\emptyset$-definable by Theorem~\ref{thmOTDefinable}.
	
	Now assume we are in the group case.
	Suppose there exists an $x_0\in \cM_G$ such that $x_0^{-1}\cdot \cO_G\subseteq G$. 
	Then $x_0^{-1}\cdot \cO_G$ is a fractional $\cO_G$-ideal  and therefore there exists a maximal fractional $\cO_G$-ideal $\cA$ with $\cA\subseteq G$. We have $\cO_G\subsetneq x_0^{-1}\cdot \cO_G\subseteq \cA$.
	Let
	$\cA_\alpha :=\left\{x\in K\mid v_G\left(x\right)\geq\alpha\cdot v_G\left(y\right) \textrm{ for some } y\in \cA\right\}$. 
	Let $x\in \cA$. If
	$v_G(x)\geq 0 =\alpha\cdot v_G\left(1\right)$, then $x\in \cA_\alpha$. 
	If $v_G\left(x\right)<0$, then $v_G\left(x\right)>\alpha\cdot v_G\left(x\right)$ and therefore $x\in \cA_\alpha$. 
	Hence $\cA\subseteq \cA_\alpha$. 
	Assume for all $x_1\in \cA\setminus \cO$ there exists $x_2\in \cA$ such that  $\left(1+p^{-1}\right)\cdot v_G(x_1)\geq v_G(x_2)$. 
	Define $\fp:=\left\{x\in K\mid -v_G(x)<v_G(a) \textrm{ for all } a\in \cA\right\}$.
	Let $a\in \cA\setminus \cO\neq \emptyset$. Then for all $x\in \fp$  $v_G(x)>-v_G(a)>0$ and hence
	$x\in \cM$. As further $a^{-1}\in \cM\setminus \fp$ we have $\fp\subsetneq \cM$. 
	Let $x,y\in \fp$. Then
	$-v_G(x+y)\leq -\max\{v_G(x),v_G(y)\}<v_G(a)$ for all $a\in \cA$. Hence $x+y\in \fp$.
	Let $x\in \fp$ and $k\in \cO$. For all $a\in \cA$ we have $k\cdot a\in \cA$. Hence $-v_G(x)>v_G(k\cdot a)$ and therefore $v_G(a)<-v_G(k\cdot x)$. Thus $k\cdot x\in \fp$.  
	Let $x,y\in \cO\setminus \fp$. Let $a,b\in \cA$ such that $-v_G(x)\geq v_G(a)$ and $-v_G(y)\geq v_G(b)$. 
	
	If  $a\in \cO$ or $b\in \cO$ we have $a\cdot b\in \cA$. As $-v_G(x\cdot y)\geq v_G(a)+v_G(b)$ we have $x\cdot y\notin \fp$. 
	
	If $a,b\in \cA\setminus \cO$  let $a_0\in\{a,b\}$ such that $v_G(a_0)=\min \{v_G(a),v_G(b)\}\in \cA\setminus \cO$. By assumption there exists $a_1\in \cA$ such that $0>\left(1+p^{-1}\right)\cdot v_G(a_0)\geq v_G(a_1)$. Recursively for all $n\geq 0$ we can define $a_{n+1}\in \cA\setminus \cO$ with $\left(1+p^{-1}\right)\cdot v_G(a_n)\geq v_G(a_{n+1})$. We then get $\left(1+p^{-1}\right)^n\cdot v_G(a_0)\geq \left(1+p^{-1}\right)^{n-1}\cdot v_G(a_1)\geq \ldots \geq v_G(a_n)$. As $\left(1+p^{-1}\right)^n\longrightarrow \infty$ for $n\rightarrow \infty$, for some $n\in \N$ we have $\left(1+p^{-1}\right)^n\geq 2$ and thus $2\cdot v_G(a_0)\geq \left(1+p^{-1}\right)^n\cdot v_G(a_0)\geq v_G(a_n)$. 
	Hence $-v_G(x\cdot y)\geq v_G(a)+v_G(b)\geq v_G(a_n)$. As $a_n\in \cA$ from this follows $x\cdot y\notin \fp$. 
	
	Altogether we see that for all $x, y\in \cO$ if $x\cdot y\in \fp$ then $x\in \fp$ or $y\in \fp$. 
	
	Hence $\fp$ is a prime ideal. Therefore $\cO_\fp$ is a proper coarsening of $\cO$. 
	
	Let $x\cdot y^{-1}\in \cO_{\fp}$. As $y\notin \fp$ there exists $a\in \cA$ such that $v_G(a)\leq -v_G(y)$. We therefore have  $v_G(x\cdot y^{-1})\geq v_G(x)+v_G(a)\geq v_G(a)$. Hence by Remark~\ref{remfracIdiff}~(a) we have $x\cdot y^{-1}\in \cA$. Thus $\cO_G\subsetneq\cO_{\fp}\subseteq \cA\subseteq G$. But this contradicts the definition of $\cO_G$.
	Hence for some $x_0\in \cA\setminus \cO_G\neq \emptyset$ we have $\left(1+p^{-1}\right)\cdot v_G\left(x_0\right)<v_G(x)$ for all $x\in \cA$.
	As $\cA\subseteq G$, there exists $y_0\in K$ such that $x_0=y_0^p-y_0$. We have $0>v_G(x_0)=v_G(y_0^p-y_0)=p\cdot v_G(y_0)$. Therefore $v_G(y_0)=p^{-1} \cdot v_G(x_0)$ and hence $v_G(x_0\cdot y_0)= \left(1+p^{-1}\right)\cdot v_G(x_0)$.
	As $\left(1+p^{-1}\right)\cdot v_G(x_0)<v_G(x)$ for all $x\in \cA$, from this follows $x_0\cdot y_0\notin \cA$. As $\left(1+p^{-1}\right)\leq \alpha$ and $v_G(x_0)<0$, we have $v_G(x_0\cdot y_0) \geq \alpha\cdot v_G(x_0)$ and hence $x_0\cdot y_0\in \cA_\alpha\setminus \cA$. This shows $\cA\subsetneq \cA_\alpha$. 
	
	Let $x\in \cA_\alpha \setminus \cA$. 
	Then there exists $y\in \cA$ such that 
	\begin{equation}\label{eqvG(x)>alphavG(y)}
	v_G\left(x\right)> \alpha\cdot v_G\left(y\right).
	\end{equation}
	As $\cO_G\subseteq \cA$ we have $v_G\left(x\right)<0$. Hence $v_G\left(y\right)<0$.
	As $x\notin \cA$ by Remark~\ref{remfracIdiff}
	$v_G\left(x\right)<v_G\left(y\right)$ and therefore
	$v_G\left(x\cdot y^{-1}\right)<0$.
	Further
	$v_G\left(x\cdot y^{-1}\right)
	> \alpha\cdot v_G\left(y\right)-v_G\left(y\right)
	>v_G\left(y\right)$
	as $\alpha-1\in \,\left(0, 1\right)$ and $v_G\left(y\right)<0$. Hence
	\begin{equation}\label{eq0>vG(y)}
	0>v_G\left(x\cdot y^{-1}\right)>v_G\left(y\right).
	\end{equation}
	Again by Remark~\ref{remfracIdiff} we get $x\cdot y^{-1}\in \cA\setminus \cO_G$.
	As $\cA\subseteq G$ there exists $a\in K$ such that $x\cdot y^{-1}=a^p-a$.
	As  $0<v_G\left(x\cdot y^{-1}\right)$  we have $v_G\left(a\right)<0$ and hence 
	\begin{equation}\label{eqvG(xy-1)=vG(ap)}
	v_G\left(x\cdot y^{-1}\right)=v_G\left(a^p\right).
	\end{equation}
	Therefore $x\cdot \left(y\cdot a^p\right)^{-1}\in \cO_G^{\times}\subseteq \cA$. 
	As $y \in \cA$ we have
	$x \cdot a^{-p}\in \cA\subseteq G $. Hence there exists $b\in K$ such that $x\cdot a^{-p}=b^p-b$. Hence
	\begin{equation}\label{x=abp-ab+ab-apb}
	x=a^p\cdot b^p-a^p\cdot b
	=\left(a\cdot b\right)^p-a\cdot b+a\cdot b-a^p\cdot b.
	\end{equation}
	We have $\left(a\cdot b\right)^p-a\cdot b\in  G$. Further
	\begin{equation}\label{eqpvG(b)=vG(y)}
	\begin{split}
	\min \left\{p\cdot v_G\left(b\right),\  v_G\left(b\right)\right\}&\ =\ v_G(b^p-b)\\
	&\ =\ v_G\left(x\right)-v_G\left(a^p\right)\\
	&\stackrel{\textrm{\scriptsize(\ref{eqvG(xy-1)=vG(ap)})}}{=}v_G(x)-v_G(x\cdot y^{-1})\\
	&\ =\ v_G\left(y\right)\\
	&\ <0.
	\intertext{Hence}
	p\cdot v_G\left(b\right)&\ =\ \min\left\{p\cdot v_G\left(b\right), v_G\left(b\right)\right\}\ =\ v_G\left(y\right).
	\end{split}
	\end{equation}
	Further as $1<\alpha\leq 2-p^{-1}$
	\begin{eqnarray*}
		v_G\left(a^p\cdot b\right)
		&\stackrel{\textrm{\scriptsize{(\ref{eqvG(xy-1)=vG(ap)})}}}{=}&v_G\left(x\cdot y^{-1}\right)+v_G\left(b\right)\\
		&\stackrel{\textrm{\scriptsize{(\ref{eqpvG(b)=vG(y)})}}}{=}&v_G\left(x\right)-v_G\left(y\right)+p^{-1} \cdot v_G\left(y\right)\\
		&\stackrel{\textrm{\scriptsize{(\ref{eqvG(x)>alphavG(y)})}}}{>}& \alpha\cdot v_G\left(y\right)-v_G\left(y\right)+p^{-1} \cdot v_G\left(y\right)\\
		&=&\left(\alpha-1+p^{-1}\right) \cdot v_G\left(y\right)\\
		&\stackrel{\textrm{\scriptsize{(\ref{eq0>vG(y)})}}}{\geq}& \left(2-p^{-1}-1-p^{-1}\right) \cdot v_G\left(y\right)\\
		&=&v_G\left(y\right).
	\end{eqnarray*}
	Therefore as $y\in \cA$ again with Remark~\ref{remfracIdiff} follows $a^p\cdot b\in \cA\subseteq G$.
	As $p>1$ and $v_G\left(a\right)<0$ we get $v_G\left(a\cdot b\right)>v_G\left(a^p\cdot b\right)$. Therefore  with Remark~\ref{remfracIdiff} follows $a\cdot b\in \cA\subseteq G$.
	As $G$ is closed under addition $x\stackrel{\textrm{\scriptsize{(\ref{x=abp-ab+ab-apb})}}}{=}\left(a\cdot b\right)^p-a\cdot b+a\cdot b-a^p\cdot b\in G$. 
	
	Hence $\cA_\alpha\setminus \cA\subseteq G$ and therefore $\cA\subseteq\cA_\alpha\subseteq G$.
	But this contradicts the maximality of $\cA$ and therefore for all $x\in \cM_G$ we have $x^{-1}\cdot \cO_G\not\subseteq G$. Hence by Theorem~\ref{thmOTDefinable} $\cO_G$  is $\emptyset$-definable.
\end{proof}

\begin{prop}\label{propOGDefinableGroupandResCase}
	Let $q\in \mathbb{N}$ be prime. Let $K$ be a field 
	with $\textrm{char}\left(K\right)\neq q$ and the $q$th-root of unity $\zeta_q\in K$. 
	Let $G:=(K^\times)^q$. Assume we are in the group case or we are in the residue case and $\overline{G}\cup\{0\}$ is no ordering on $K$.
	Then $\cO_G$ is $\emptyset$-definable.
	In particular $\cO_G$ is $\emptyset$-definable if $\ch(K)>0$.
\end{prop}

\begin{proof}
	The case $\cO_G= K$ is clear.
	
	If $\cO_G\neq K^\times$ we have $G\neq K^\times$ and hence the claim follows  
	by Theorem~\ref{thmOTDefinable}.
	
	Further by Lemma~\ref{lemWeaklyCompIffComp} if $q\neq \ch(K)>0$ the weak case can not occur.
\end{proof}

\begin{prop}\label{propOG-Gdefinableresidueordering}
	Let $K$ be a field 
	with $\textrm{char}\left(K\right)=0$ and $\zeta_2\in K$. 
	Let $K$ not be euclidean, i.e. $K^2$ is not a positive cone.
	Let $G:=(K^\times)^2$. Let $\overline{G}\cup \{0\}$ be an ordering on $\overline{K}$. 
	Then $\cO_{G\cup(-G)}$ is $\emptyset$-definable.
	Further if $\cO_G$ is non-trivial then it induces the same topology as $\cO_{G\cup(-G)}$.
\end{prop}

\begin{proof}
	As $G$ is a subgroup of $K$ it is easy to see that $G\cup\left(-G\right)$ is a subgroup of $K$ as well.
	As $\overline{K}$ is real, $K$ is real as well (see \cite[Corollary~2.2.6]{EnPr2005}).
	Suppose $ K^\times= G\cup \left(-G\right)$. 
	Then $ K=K^2\cup\left( -K^2\right)$.
	It is clear that $K^2\cdot K^2\subseteq K^2$, $K^2\subseteq K^2$ and $-1\notin \sum K^2$.
	Suppose $K^2+K^2\not \subseteq K^2$. Hence there exist $x, y\in K$ such that $x^2+y^2\notin K^2$. 
	By assumption $K=K^2\cup\left(-K^2\right)$ and therefore
	$x^2+y^2\in -K^2$. Thus $x^2\cdot\left(-\left(x^2+y^2\right)\right)^{-1}, y^2\cdot\left(-\left(x^2+y^2\right)\right)^{-1}\in K^2$ and hence $-1=x^2\cdot\left(-\left(x^2+y^2\right)\right)^{-1}+y^2\cdot\left(-\left(x^2+y^2\right)\right)^{-1}\in\sum K^2$. But this is a contradiction to $K$ real.
	From this follows that $K^2$ is a positive cone. As we assumed that $K$ is not euclidean this is a contradiction and hence $K^2\cup \left(-K^2\right)\neq K$. Thus $G\cup \left(-G\right)\neq K^\times$.
	By Lemma~\ref{lemWeaklyCompIffComp} we are not in the weak case. Let $x\in \cO_{G\cup\left(-G\right)}^\times$. If $x\notin G$ by Lemma~\ref{lemOverlineGandG}~(a) 
	follows $\overline{x}\notin \overline{G}$ and therefore  $\overline{-x}\in \overline{G}$. Again by Lemma~\ref{lemOverlineGandG}~(a) 
	$-x\in G$ and thus $x\in -G$. 
	Hence $\cO_{G\cup\left(-G\right)}^\times\subseteq G\cup \left(-G\right)$. Therefore we are in the group case  and $\cO_{G\cup\left(-G\right)}$ is $\emptyset$-definable by Theorem~\ref{thmOTDefinable}. 
	
	Assume $\cO_G$ is non-trivial. As $1+\cM_G\subseteq G\subseteq G\cup \left(-G\right)$, $\cO_G$ is compatible with $G\cup \left(-G\right)$. Therefore by Lemma~\ref{lemNonTrivialIFF} $\cO_{G\cup\left(-G\right)}$ is non-trivial. As $\cO_{G\cup\left(-G\right)}$ is as well compatible with $G\cup\left(-G\right)$,  $\cO_{G\cup\left(-G\right)}$ induces the same topology as $\cO_G$.
\end{proof}

We will generalize the notion of henselianity slightly and define when a valued field is called $q$-henselian for a certain prime $q$. We denote by $K\left\langle q\right\rangle$ the compositum of all finite Galois extensions of $q$-power degree. 
$(K,\,\cO)$ is \emph{$q$-henselian} if $\cO$ extends uniquely $K\left\langle q\right\rangle$.

\begin{prop}\label{proppHensIff3cases}
	Let $\left(K,\,v\right)$ be a valued field, let $q$ be prime and and if $q\neq \ch(K)$ assume  $\zeta_q\in K$.
	\begin{enumerate}[(a)]
		\item\label{proppHensIff3casesqneqcharEqual} If $\textrm{char}\left(\overline{K}\right)\neq q$, then $v$ is $q$-henselian if and only if $1+\cM_v\subseteq \left(K^{\times}\right)^q$.
		\item\label{proppHensIff3casesqeqchar}If $\textrm{char}\left(K\right)= q$, then $v$ is $q$-henselian if and only if $\cM_v\subseteq K^{(q)}$.
		\item\label{proppHensIff3casesqneqcharMixedChar} If $\textrm{char}\left(K\right)= 0$, $\textrm{char}\left(\overline{K}\right)= q$ and $v$ is a rank-1-valuation, then $v$ is $q$-henselian if and only if $1+q^n\cdot \cM_v\subseteq \left(K^{\times}\right)^q$ for some $n\in \N$. In this case $1+q^n\cdot \cM_v\subseteq \left(K^{\times}\right)^q$ for every $n\geq 2$.
	\end{enumerate}
\end{prop}

Proposition~\ref{proppHensIff3cases} is essentially \cite[Proposition~1.4]{Ko1995},  assertion~(\ref{proppHensIff3casesqneqcharMixedChar}) is slightly adjusted as in \cite{Ko1995} this is only shown  for $n=2$. 
As the proof works the same way, we will not repeat it here (for details see \cite[Proposition~5.10]{Du2015}). The original proof Assertion~(\ref{proppHensIff3casesqeqchar}) has a gap. For a corrected proof see~\cite{ChPe2015}.

\begin{prop}\label{propqhenselianCompatible}
	Let $\left(K,\,v\right)$ be a valued field, let $q$ be prime such that $v$ is $q$-henselian.
	\begin{enumerate}[(a)]
		\item Let $\textrm{char}\left(K\right)=p=q$ and $G:= K^{(p)}$. Then $v$ is compatible.                                                                    
		\item Let $\textrm{char}\left(\overline{K}\right)\neq q$, $\zeta_q\in K$ and  $G:= \left(K^{\times}\right)^q$.  Then $v$ is compatible.
		\item Let $\textrm{char}\left(K\right)= 0$, $\textrm{char}\left(\overline{K}\right)= q$ , $\zeta_q\in K$ and  $G:= \left(K^{\times}\right)^q$. Then $1+q^2\cdot \cM_v\subseteq G$. If further $v$ is a rank-1 valuation, then $v$ is weakly compatible. 
	\end{enumerate}
\end{prop}

\begin{proof}
	Assertion (a) and (b) and the first part of (c) follow at once from Proposition~\ref{proppHensIff3cases}
	
	Now assume $\textrm{char}\left(K\right)= 0$, $\textrm{char}\left(\overline{K}\right)= q$, $\zeta_q\in K$,  $G:= \left(K^{\times}\right)^q$ and $v$ is of rank-1.
	Let $\cA=q^2\cdot \cM_v=\left\{x\in K\mid v(x)>v\left(q^2\right)\right\}$. $\cA$ is an $\cO_v$-ideal. As $v$ is of rank-1,  $\Gamma$ is archimedean. Hence  for every $x\in \cM_v$ there exists $n\in \N$ with $v\left(x^n\right)>v\left(q^2\right)$ and thus $x^n\in \cA$. Therefore $\sqrt{\cA}=\cM_v$. As $1+\cA\subseteq \cM_v$, it follows that $v$ is weakly compatible. 
\end{proof}

\begin{prop}\label{propOGphenselian}
	Let $K$ be a valued field and let $p=\ch(K)>0$. Let $G:=K^{(p)}$.
	Then $\cO_G$ is $p$-henselian.
\end{prop}

\begin{proof}
	By Lemma~\ref{lemWeaklyCompIffComp} $\cO_G$ is compatible. Hence $\cM_G\subseteq G=K^{(p)}$. By Proposition~\ref{proppHensIff3cases}~(\ref{proppHensIff3casesqeqchar}) $\cO_G$ is $p$-henselian.
\end{proof}

\begin{prop}\label{propOGqhenseliancoarsening}
	Let $K$ be a valued field, let $q\neq\ch(K)$ be prime and $\zeta_q\in K$. Let $G:=\left(K^\times\right)^q$.
	\begin{enumerate}[(a)]
		\item If $\textrm{char}\left(\cO_G/\cM_G\right)\neq q$, then $\cO_G$ is $q$-henselian 
		\item If $\textrm{char}\left(K\right)= 0$ and $\textrm{char}\left(\cO_G/\cM_G\right)= q$, then $\cO_G$ has a non-trivial $q$-henselian coarsening.
	\end{enumerate}
\end{prop}

\begin{proof}
	\begin{enumerate}[(a)]
		\item By Lemma~\ref{lemWeaklyCompIffComp} $\cO_G$ is compatible. Hence $1+\cM_G\subseteq G=\left(K^\times\right)^q$. By Proposition~\ref{proppHensIff3cases}~(\ref{proppHensIff3casesqneqcharEqual}) $\cO_G$ is $p$-henselian.
		\item By Proposition~\ref{propMaximalValuationorBasis}, either there exists a maximal non\hbox{-}trivial coarsening of $\cO_G$ or the non-zero prime ideals of $\cO$ form a basis of the neighbourhoods of zero of the topology $\cT_\cO$. 
		
		Let us first assume $\widetilde{\cO}$ is a maximal non-trivial coarsening of $\cO_G$. Then $\widetilde{\cO}$ has rank-1. Let $\widetilde{\cM}$ denote the maximal ideal of $\widetilde{\cO}$. As $\cO_G$ is coarsely compatible, so is $\widetilde{\cO}$ and hence by Lemma~\ref{lemWeaklyCompthen1+pnMvsubsetT} there exists $n\in \N$ with $1+q^n \cdot \widetilde{\cM}\subseteq G= (K^\times)^q$.
		If $\ch\left(\widetilde{\cO}/\widetilde{\cM}\right)=q$ then by Proposition~\ref{proppHensIff3cases}~(\ref{proppHensIff3casesqneqcharMixedChar}) $\widetilde{\cO}$ is $q$-henselian. If $\ch\left(\widetilde{\cO}/\widetilde{\cM}\right)=0$ then $1+\widetilde{\cM}=1+q^n \cdot \widetilde{\cM}\subseteq (K^\times)^q$ and hence by Proposition~\ref{proppHensIff3cases}~(\ref{proppHensIff3casesqneqcharEqual}) $\widetilde{\cO}$ is $q$-henselian. 
		
		Now assume the non-zero prime ideals of $\cO$ form a basis of the neighbourhoods of zero of  $\cT_\cO$. Then there exists an $\cO_G$-prime ideal $\fp\neq \{0\}$ such that $q\notin \fp$. $\widetilde{\cO}:=\left(\cO_G\right)_\fp$ is a proper coarsening of $\cO_G$ with maximal ideal $\widetilde{\cM}:=\fp\subsetneq \cM_G$. As $1+\widetilde{\cM}\subseteq G= (K^\times)^q$ and $\ch\left(\widetilde{\cO}/\widetilde{\cM}\right)=0$  by Proposition~\ref{proppHensIff3cases}~(\ref{proppHensIff3casesqneqcharEqual}) $\widetilde{\cO}$ is $q$-henselian. 
	\end{enumerate}
\end{proof}

Similar as the canonical henselian valuation (see \cite[Section~4.4]{EnPr2005}) we can define the canonical 
$q$-henselian valuation. (See \cite[Section~2]{JaKo2014} for details):
\begin{lem}
	Let $q$ be prime. Let $K$ be a field which is not $q$-closed. We divide the class of $q$-henselian valuations into two subclasses,
	\[
	H_1^q(K):=\left\{v\mid v \textrm{ is a q henselian valuation and }\overline{K}_v\neq \overline{K}_v\langle q\rangle\right\}
	\]
	and
	\[
	H_2^q(K):=\left\{v\mid v \textrm{ is a q henselian valuation and }\overline{K}_v= \overline{K}_v\langle q\rangle\right\}.
	\]
	If $H_2^q(K)\neq \emptyset$ then there exists a unique coarsest valuation $v_K^q\in H_2^q(K)$. 
	
	Otherwise there exists a unique finest valuation $v_K^q\in H_1^q(K)$.	
\end{lem}

\begin{de}
	We call $v_K^q$ the canonical $q$-henselian valuation.
\end{de}

\begin{rem}
	Note that $v_K^q$ is the trivial valuation if and only if $K$ admits no non-trivial $q$-henselian valuation or $K=K\langle q\rangle$.
\end{rem}

\begin{thm}\label{thmJaKoCanqHensDef}
	Let $K$ be a field which is not $q$-closed.
	Let $\mathrm{char}\left(K\right)\neq q$, $\zeta_q\in K$ and if $q= 2$ assume the residue field of the canonical henselian valuation $\cO_{v_K^q}/\cM_{v_K^q}$ is not euclidean.
	Then $v_K^q$ is $\emptyset$-definable.
\end{thm}

Theorem~\ref{thmJaKoCanqHensDef} is a simplified version of \cite[Main~Theorem~3.1]{JaKo2014} omitting some details we will not need.

\begin{prop}\label{propDefinableIfResFieldEuclidean}
	Let $K\neq K\langle 2\rangle$ and assume $\cO_{v_K^q}/\cM_{v_K^q}$ is euclidean. 
	Then the coarsest $2$-henselian valuation ${v_K^{2}}^*$ on $K$ which has a euclidean residue field is
	$\emptyset$-definable. 
\end{prop}
Proposition~\ref{propDefinableIfResFieldEuclidean} is  \cite[Observation~3.2~(a)]{JaKo2014}.

The following proposition is in particular interesting in the weak case, where $\cO_G$ in general is not definable.

\begin{prop}\label{propDefinableWeakCase}
	Let $q\in \mathbb{N}$ be prime. Let $K$ be a field 
	with $\textrm{char}\left(K\right)=0$ and $\zeta_q\in K$. 
	Let $G:=(K^\times)^q$. Assume that we are in the weak case.
	Then $K$ admits a $q$-henselian $\emptyset$-definable valuation which induces the same topology as $\cO_G$.
\end{prop}

\begin{proof}
	The case $\cO_G= K$ is clear. Hence assume $\cO_G\neq K$ and therefore $G\neq K^\times$.
	As we are in the weak case, by Lemma~\ref{lemWeaklyCompIffComp} $\ch(\cO_G/\cM_G)=q$. 
	By Proposition~\ref{propOGqhenseliancoarsening} $\cO_G$ has a non-trivial $q$-henselian coarsening. By Theorem~\ref{thmJaKoCanqHensDef} and Proposition~\ref{propDefinableIfResFieldEuclidean} either $v_K^{q}$ or  ${v_K^{2}}^*$ is $\emptyset$-definable, non-trivial and induces the same topology as $\cO_G$. 	
	
	By Lemma~\ref{lemWeaklyCompIffComp} the weak case can only occur if $\textrm{char}\left(\cO_G/\cM_G\right)=q$.
\end{proof}

\begin{thm}\label{thmOGnontrivialdefinable}
	Let $K$ be a field.
	\begin{itemize}
		\item Let $\ch(K)=q$ and $G:=K^{(q)}$ \textbf{ or }
		\item let $\mathrm{char}\left(K\right)\neq q$, $\zeta_q\in K$,   $G:=\left(K^\times\right)^q$ and if $q= 2$ assume $K$ is not euclidean.  
	\end{itemize}	
	Assume $\cO_G$ is non-trivial.
	Then $K$ admits a non-trivial $\emptyset$-definable valuation inducing the same topology as $\cO_G$. 
\end{thm}

\begin{proof}
	If $\ch(K)=q$ let $G:=K^{(q)}$.  By Theorem~\ref{thmArtinschreierindBewDef} $\cO_G$ is $\emptyset$-definable.
	
	If $\mathrm{char}\left(K\right)\neq q$, $\zeta_q\in K$ by Proposition~\ref{propOGDefinableGroupandResCase}, Proposition~\ref{propOG-Gdefinableresidueordering} and Proposition~\ref{propDefinableWeakCase} there exists a $\emptyset$-definable valuation which induces the same topology as $\cO_G$.  
\end{proof}

\begin{cor}\label{corVAxiomsnontrivialdefinable}
	Let $K$ be a field.
	\begin{itemize}
		\item Let $\ch(K)=q$ and $G:=K^{(q)}$ \textbf{ or }
		\item let $\mathrm{char}\left(K\right)\neq q$, $\zeta_q\in K$,   $G:=\left(K^\times\right)^q$ and if $q= 2$ assume $K$ is not euclidean.  
	\end{itemize}
	Assume that for $\cN=\left\{U\in \cT_G\mid 0\in U\right\} $
	\begin{enumerate}[(V\,1)]
		\item  $\bigcap \mathcal{N}:=\bigcap_{U\in\mathcal{N}}U=\left\{0\right\}$ and $\left\{0\right\}\notin\mathcal{N}$;
		\item $\forall\, U,\,V\in\mathcal{N}\ \exists\, W\in\mathcal{N}\   W\subseteq U\cap V$;
		\item $\forall\, U\in\mathcal{N}\   \exists\,  V\in\mathcal{N}\   V-V\subseteq U$;
		\item$\forall\, U\in\mathcal{N}\   \forall\, x,\,y\in K\   \exists\, V\in\mathcal{N}\   \left(x+V\right)\cdot\left(y+V\right)\subseteq x\cdot y+U$;
		\item $\forall\, U\in \mathcal{N}\   \forall\, x\in K^{\times}\  \exists\, V\in \mathcal{N}\   \left(x+V\right)^{-1}\subseteq x^{-1}+U$;
		\item $\forall\, U\in\mathcal{N}\   \exists\, V\in \mathcal{N}\   \forall\, x,\,y\in K\   x\cdot y\in V\longrightarrow x\in U\, \vee\,  y\in U$.
	\end{enumerate}
	Then $K$ admits a non-trivial $\emptyset$-definable valuation. 
\end{cor}

This follows at once by Theorem~\ref{thmOGnontrivialdefinable} and Corollary~\ref{corOGnontrivialiffexiffVtop}.

\begin{thm}\label{thmphenseliandefinable}
	Let $K$ be a field which is not $q$-closed.
	\begin{itemize}
		\item  Let $\ch(K)=q$ \textbf{ or }
		\item let $\mathrm{char}\left(K\right)\neq q$, $\zeta_q\in K$ and if $q= 2$ assume $K$ is not euclidean.
	\end{itemize}
	Assume $K$ admits a non-trivial $q$-henselian valuation $v$.
	Then $K$ admits a non-trivial $\emptyset$-definable valuation which induces the same topology as $v$.
\end{thm}

\begin{proof}
	As $K$ is not $q$-closed $G\neq K$ $\left[\textrm{resp. }G\neq K^\times\right]$.
	
	If $\ch(K)=q$ let $G:=K^{(q)}$, otherwise let $G:=\left(K^\times\right)^q$. 
	
	If $\ch(K)=q$ or $\ch(\overline{K})\neq q$, then $v$ is weakly compatible  by Proposition~\ref{propqhenselianCompatible}. Hence by Lemma~\ref{lemNonTrivialIFF} $\cO_G$ is non-trivial.
	
	If $\ch(K)=0$ and $\ch(\overline{K})=q$ by Proposition~\ref{propMaximalValuationorBasis} either there exists a maximal non-trivial coarsening of $\cO_v$ or the non-zero prime ideals of $\cO$ form a basis of the neighbourhoods of zero of $\cT_\cO$.
	
	If $\widetilde{\cO}$ is a maximal non-trivial coarsening of $\cO_v$, then $\widetilde{\cO}$ has rank-1.  As a coarsening of a $q$-henselian valuation ring, $\widetilde{\cO}$ is $q$-henselian and hence by Proposition~\ref{propqhenselianCompatible}~(c) $\widetilde{\cO}$ is weakly compatible. Again  by Lemma~\ref{lemNonTrivialIFF} $\cO_G$ is non-trivial.
	
	If the non-zero prime ideals of $\cO$ form a basis of the neighbourhoods of zero of  $\cT_\cO$, there exists an $\cO_v$-prime ideal $\fp\neq \{0\}$ such that $q\notin \fp$. $\widetilde{\cO}:=\left(\cO_v\right)_\fp$ is a proper coarsening of $\cO_v$ with maximal ideal $\widetilde{\cM}:=\fp$. As a coarsening of a $q$-henselian valuation ring, $\widetilde{\cO}$ is $q$-henselian and hence by Proposition~\ref{propqhenselianCompatible}~(b) compatible.  Again  by Lemma~\ref{lemNonTrivialIFF} $\cO_G$ is non-trivial.

	By Theorem~\ref{thmOGnontrivialdefinable} $K$ admits a non-trivial $\emptyset$-definable valuation inducing the same topology as $\cO_G$ and hence as $v$. 
	As $\cO_v$ and $\cO_G$ are both weakly compatible,  $v$ induces the same topology as $\cO_G$ (see \cite[Theorem~2.3.4]{EnPr2005}). 
\end{proof}

\begin{rem}
	Under the assumptions of Theorem~\ref{thmphenseliandefinable}:
	\begin{enumerate}[(a)]
		\item  There exists a non-trivial $q$-henselian definable valuation inducing the same topology as $v$.
		\item If $q=\ch(K)$, $v$ induces the same topology as $\cO_G$ for $G=K^{(q)}$.
		\item If $q\neq\ch(K)$, $v$ induces the same topology as $\cO_G$ for $G=(K^\times)^{q}$.
	\end{enumerate}
\end{rem}

\begin{proof}
	\begin{enumerate}[(a)]
		\item If  $q=\ch(K)$ or if $q\neq 2$ and $q\neq \ch(\cO_G/\cM_G)$, the definable valuation in Theorem~\ref{thmphenseliandefinable} is $\cO_G$ and, by Proposition~\ref{propOGphenselian} or Proposition~\ref{propOGqhenseliancoarsening}~(a), $\cO_G$ is $q$-henselian. If $q= 2\neq \ch(\cO_G/\cM_G)$ as well by Proposition~\ref{propOGqhenseliancoarsening}~(a), $\cO_G$ is $q$-henselian. Therefore the $q$-henselian definable valuation we obtain by Theorem~\ref{thmJaKoCanqHensDef} or Proposition~\ref{propDefinableIfResFieldEuclidean}, is non-trivial. 
		If $q\neq \ch(\cO_G/\cM_G)$ with the same proof as for the weak case in Proposition~\ref{propDefinableWeakCase}, we obtain a $q$-henselian definable valuation in all cases. 
		\item By Proposition~\ref{propOGphenselian} $\cO_G$ is $q$-henselian. As $K$ is not $q$-closed any two $q$-henselian topologies are dependent and therefore $v$ induces the same topology as $\cO_G$.
		\item By Proposition~\ref{propOGqhenseliancoarsening}~(a) some coarsening $\widetilde{\cO}$ of $\cO_G$ is $q$-henselian. As $K$ is not $q$-closed any two $q$-henselian topologies are dependent and therefore $v$ induces the same topology as $\widetilde{\cO}$ and hence as $\cO_G$.
	\end{enumerate}
\end{proof}

A field with a V-topology is called t-henselian if it is locally equivalent to a field with a topology induced by a henselian valuation. For details see \cite{PrZi1978}. In particular any field with a topology induced by a henselian valuation is t-henselian. The converse is not true. An example was indicated in \cite[Page~338]{PrZi1978}, details are given in \cite[Konstruktion~5.3.5]{Du2010}.

\begin{thm}\label{thmDefValtHenselian}
	Let $\left(K,\,\cT\right)$ be a t-henselian field. 
	There exists a definable valuation on $K$ which induces the topology $\cT$ if and only if $K$ is neither real closed nor separably closed.
\end{thm}

\begin{proof}
	In archimedean ordered real closed fields for every definable set either the set itself or the complement is bounded by a natural number, which can not be true for a non-trivial valuation ring.
	As the theory of real closed fields is complete, from this follows already that no real closed field  admits a definable valuation.
	If a field admits a non-trivial definable valuation we can construct a formula with the strong order property. Hence the field is not simple and therefore in particular not separably closed. For more details see \cite[6.58-6.61]{Du2015} and \cite[Section~8.2]{TeZi2012}.
	
	If $\left(K,\,\cT\right)$ is a not real closed and not separably closed t-henselian field, it  is locally equivalent, and hence elementary equivalent, to a field  $\widetilde{K}$ with a topology induced by a henselian valuation $v$.

	$\widetilde{K}$ is as well not real closed and not separably closed, hence there exists a prime $q$ and a field $L$ such that
	$L/\widetilde{K}$ is a finite separable extension,
	$L\neq L\langle q\rangle$,
	and if $q\neq \ch(L)$ then $\zeta_q\in L$.
	Let $w$ be the unique extension of $v$ to $L$. $w$ is henselian and hence $q$-henselian.
	
	If $q=\ch(L)$ or  $q\neq \ch(L)$ and $q \neq 2$ or $L$ is not euclidean, then by Theorem~\ref{thmphenseliandefinable} there exists a definable valuation $\widetilde{w}$ on $L$ which induces the same topology as $w$.
	
	Now assume $q=2\neq \ch(L)$ and  $L$ is euclidean.	
	As $L$ is not real closed and euclidean,  there exists a polynomial $f\in L[X]$ of odd degree such that $f$ has no roots in $L$ (see for example \cite[Theorem~1.2.10 (Artin,Schreier)]{PrDe2001}).  Without loss of generality let $f$ be irreducible. 
	Let $x\in L^{\textrm{sep}}$ be a root of $f$. We have $[L(x):L]=\textrm{deg}(f)$ positive and odd. Hence there exists $\widetilde{q}\neq 2$ prime such that $\widetilde{q}\,|[L(x):L]$. Now we can find a field extension $\widetilde{L}$ such that we can prove as above for $\widetilde{L}$ and $\widetilde{q}$ that there is a definable valuation $\widetilde{w}$ on $\widetilde{L}$.

	By Proposition~\ref{propDefValOnSubfields} $\widetilde{w}|_K$ is a definable valuation on $\widetilde{K}$. As $\widetilde{w}$ induces the same topology on $L$  as $w$  it is easy to see that $\widetilde{w}|_{\widetilde{K}}$ and $w|_{\widetilde{K}}$  induce the same topology on $\widetilde{K}$.

	As  $K$ and $\widetilde{K}$ are elementary equivalent, there exists a definable valuation $v_0$ on $K$. As $\left(\widetilde{K},\, \cT_v\right)$  and $\left(K,\,\cT\right)$ are locally equivalent, follows $\cT_{v_0}=\cT$. 
\end{proof}

\end{document}